\documentclass[reqno,a4paper,twoside]{amsart}

\usepackage[a4paper,margin=0.95in,footskip=0.25in]{geometry}

\usepackage{enumitem}

\setenumerate{label=\textnormal{(\arabic*)}}

\allowdisplaybreaks[4]

\usepackage{amsmath,amssymb,dsfont,verbatim,bm,geometry,mathrsfs}

\usepackage{mathtools}
\usepackage[raggedright]{titlesec}

\usepackage[final]{microtype}

\usepackage[utf8]{inputenc}
\usepackage[T1]{fontenc}

\usepackage{hyperref}
\hypersetup{
colorlinks=true,
linkcolor=blue,
filecolor=magenta,
urlcolor=cyan,
unicode=true,
pdfpagemode=FullScreen,
}

\titleformat{\chapter}[display]
{\normalfont\huge\bfseries}{\chaptertitlename\\thechapter}{20pt}{\Huge}
\titleformat{\section}
{\normalfont\Large\bfseries\center}{\thesection}{1em}{}
\titleformat{\subsection}
{\normalfont\large\bfseries}{\thesubsection}{1em}{}
\titleformat{\subsubsection}
{\normalfont\normalsize\bfseries}{\thesubsubsection}{1em}{}
\titleformat{\paragraph}[runin]
{\normalfont\normalsize\bfseries}{\theparagraph}{1em}{}
\titleformat{\subparagraph}[runin]
{\normalfont\normalsize\bfseries}{\thesubparagraph}{1em}{}

\titlespacing*{\chapter} {0pt}{50pt}{40pt}
\titlespacing*{\section} {0pt}{3.5ex plus 1ex minus .2ex}{2.3ex plus .2ex}
\titlespacing*{\subsection} {0pt}{3.25ex plus 1ex minus .2ex}{1.5ex plus .2ex}
\titlespacing*{\subsubsection}{0pt}{3.25ex plus 1ex minus .2ex}{1.5ex plus .2ex}
\titlespacing*{\paragraph} {0pt}{3.25ex plus 1ex minus .2ex}{1em}
\titlespacing*{\subparagraph} {\parindent}{3.25ex plus 1ex minus .2ex}{1em}

\usepackage{tikz}
\usetikzlibrary{matrix, shapes}

\usepackage{float, subfig}
\usepackage{amsrefs}

\usepackage{titletoc}

\contentsmargin{2.55em}
\dottedcontents{section}[1.5em]{}{2em}{1pc}
\dottedcontents{subsection}[4.35em]{}{2.8em}{1pc}
\dottedcontents{subsubsection}[7.6em]{}{3.2em}{1pc}
\dottedcontents{paragraph}[10.3em]{}{3.2em}{1pc}

\usepackage[columns=3,rule=0pt]{idxlayout}
\setindexprenote{For convenience of the reader we list below almost all symbols used in this work, indicating the page in which each of them is introduced. We only except those that are extremely standard.}

\makeindex

\newtheorem{theorem}{Theorem}[section]
\newtheorem{lemma}[theorem]{Lemma}
\newtheorem{proposition}[theorem]{Proposition}

\newtheorem{thmx}{Theorem}

\theoremstyle{definition}

\theoremstyle{remark}
\newtheorem{remark}[theorem]{Remark}

\usepackage{tikz}
\usetikzlibrary{matrix,shapes}
\usetikzlibrary{cd}

\DeclareMathOperator{\ide}{id}

\DeclareMathOperator{\RH}{RH}

\DeclareMathOperator{\Ss}{S}

\DeclareMathOperator{\ima}{Im}

\DeclareMathOperator{\Ho}{H}

\DeclareMathOperator{\Hom}{Hom}

\DeclareMathOperator{\Tot}{Tot}

\DeclareMathOperator{\RC}{RC}
\DeclareMathOperator{\sh}{sh}

\DeclareMathOperator{\sg}{sg}
\DeclareMathOperator{\Tor}{Tor}

\newcommand{\xcirc}{\hspace{0.9pt}}

\newcommand{\ov}{\overline}
\newcommand{\ot}{\otimes}
\newcommand{\wh}{\widehat}

\newcommand{\ep}{\epsilon}

\newcommand{\bx}{\mathbf x}

\newcommand{\byy}{\mathbf y}

\numberwithin{equation}{section}

\DeclareMathAlphabet{\mathpzc}{OT1}{pzc}{m}{it}

\usepackage{mathtools}
\newtagform{brackets}{[}{]}
\usetagform{brackets}

\begin{document}

\title{(Co)homology of Some Cyclic Linear Cycle Sets}

\author{Jorge A. Guccione}
\address{Departamento de Matem\'atica\\ Facultad de Ciencias Exactas y Naturales-UBA, Pabell\'on~1-Ciudad Universitaria\\ Intendente Guiraldes 2160 (C1428EGA) Buenos Aires, Argentina.}
\address{Instituto de Investigaciones Matem\'aticas ``Luis A. Santal\'o''\\ Pabell\'on~1-Ciudad Universitaria\\ Intendente Guiraldes 2160 (C1428EGA) Buenos Aires, Argentina.}
\email{vander@dm.uba.ar}

\author{Juan J. Guccione}
\address{Departamento de Matem\'atica\\ Facultad de Ciencias Exactas y Naturales-UBA\\ Pabell\'on~1-Ciudad Universitaria\\ Intendente Guiraldes 2160 (C1428EGA) Buenos Aires, Argentina.}
\address{Instituto Argentino de Matem\'atica-CONICET\\ Saavedra 15 3er piso\\ (\!C1083ACA\!) Buenos Aires, Argentina.}
\email{jjgucci@dm.uba.ar}


\begin{abstract} For each member $\mathcal{A}$ of a family of linear cycle sets whose underlying abelian group is cyclic of order a power of a prime number, we compute all the central extensions of $\mathcal{A}$ by an arbitrary abelian group.
\end{abstract}

\maketitle

\tableofcontents

\section*{Introduction}\label{section:Introduccion}

A {\em cycle set}, as defined in \cite{R1}, is a set $A$ endowed with a binary operation $\cdot$, such that the left translations $a\mapsto b\cdot a$ are bijective and the identities
$$
(a\cdot b)\cdot (a\cdot c) = (b\cdot a)\cdot (b\cdot c)
$$
are satisfied. In \cite{R1} it was proved that {\em non-degenerate cycle sets} (i.e., with invertible squaring map $a\mapsto a\cdot a$) are in bijective correspondence with non-degenerate involutive set-theoretic solutions of the Yang-Baxter equation, whose study was started by Etingof, Schedler, and Soloviev in \cite{ESS}. These solutions are connected with many domains of algebra: Garside structures, Hopf-Galois theory, affine torsors, Artin-Schelter regular rings, groups of $I$-type, left symmetric algebras, etcetera (see, for instance \cites{CJO, DG, De1, DDM, Ch, GI1, GI2, GI3, GI4, GIM, GIVB, JO, R2, R3}). A {\em linear cycle set} is a cycle set $(A,\cdot)$ endowed with an abelian group operation $+$ satisfying the identities
$$
a\cdot(b+c) = a\cdot b + a\cdot c\quad\text{and}\quad (a+b)\cdot c = (a\cdot b)\cdot (a\cdot c).
$$
The interest in these structures is due to the fact that they are equivalent to brace structures, and so they are strongly related with the non-degenerated involutive set theoretic solutions of the Yang-Baxter equations. For instance, the structure group of a non-degenerate solution \cite{ESS} is a brace in a natural way.

Motivated by the problem of the classify braces, in \cite{BCJO} the authors point out the importance of to develop a extensions theory of braces (or equivalently, of linear cycle sets). This was made out by Bachiller in \cite{B}, using the language of braces; by Ben David and Ginosar in \cite{BG}, using the language of bijective $1$-cocycles (other avatar of linear cycle sets); and by Lebed and Vendramin in \cite{LV}, using the language of linear cycle sets. In the last approach the authors introduce a cohomology theory $\Ho_N^*(\mathcal{A},\Gamma)$, in order to classify the central extensions of a linear cycle set $\mathcal{A}=(A,+,\cdot)$ by an abelian group $\Gamma$. This cohomology is defined by using a explicit cochain complex $\bigl(C_N^*(\mathcal{A},\Gamma),\partial^*\bigr)$. This allows to use homological methods in order to studied these extensions. To be something more precise, when $\mathcal{A}$ is a group $(A,+)$ endowed with the trivial linear cycle set structure $a\cdot b\coloneqq b$, one can use resolutions, satellite functors, simplicial methods, etcetera, to make calculations and to obtain theoretical results about $\Ho_N^*(\mathcal{A},\Gamma)$; and it is reasonable to expect that, under right circumstances, these calculations and results can be extended to more general types linear cycle sets. For example, this occurs if the necessary hypotheses to apply the Perturbation Lemma are satisfied.

\smallskip

Let $p\in \mathds{N}$ be a prime number and let $\nu,\eta\in \mathds{N}$ be such that $0<\nu\le \eta\le 2\nu$. Let $u\coloneqq p^{\nu}$, $v\coloneqq p^{\eta}$, $t\coloneqq p^{\eta-\nu}$, $u'\coloneqq p^{2\nu-\eta}$ and let $\mathcal{A}$ be the linear cycle set $(\mathds{Z}/v\mathds{Z};\cdot)$, where $\imath\cdot \jmath\coloneqq (1-u\imath)\jmath$. Let $\Gamma$ be an additive abelian group. The main result of this paper are the following:

\begin{thmx}
Assume that $u=v$. For each $\gamma,\gamma_1\in \Gamma$ such that $v\gamma_1=0$, let $\ov{\xi}^1_{\gamma},\ov{\xi}^2_{\gamma_1}\colon \frac{\mathds{Z}} {v\mathds{Z}}\times \frac{\mathds{Z}}{v\mathds{Z}}\to \Gamma$ be the maps defined by
$$
\ov{\xi}^1_{\gamma}(\imath_1,\imath_2)\coloneqq \xi^1_{\gamma}([g^{\imath_1}\ot g^{\imath_2}])\quad\text{and}\quad \ov{\xi}^2_{\gamma_1}(\imath_1,\imath_2)\coloneqq \xi^2_{\gamma_1}(g^{\imath_1}\ot g^{\imath_2}),
$$
where $\xi^1_{\gamma}$ and $\xi^2_{\gamma_1}$ are as above of Proposition~\ref{complemento 1}. The following facts hold:

\begin{enumerate}

\item $\Gamma\times  \frac{\mathds{Z}}{v\mathds{Z}}$ is a linear cycle set via
$$
(c,\imath)+(c',\imath') \coloneqq \bigl(c+c'+\ov{\xi}^1_{\gamma}(\imath,\imath'),\imath+\imath'\bigr)\quad\text{and}\quad (c,\imath)\cdot (c',\imath') \coloneqq \bigl(c'+\ov{\xi}^2_{\gamma_1}(\imath,\imath'),\imath\cdot \imath'\bigr).
$$
Following \cite{LV} we denoted this linear cycle set by $\Gamma\oplus_{\ov{\xi}^2_{\gamma_1},\ov{\xi}^1_{\gamma}} \frac{\mathds{Z}}{v\mathds{Z}}$. Moreover
$$
\begin{tikzpicture}
\begin{scope}[yshift=0cm,xshift=0cm]
\matrix(BPcomplex) [matrix of math nodes, row sep=0em, text height=1.5ex, text
depth=0.25ex, column sep=2.5em, inner sep=0pt, minimum height=5mm, minimum width =6mm]
{0 & \Gamma & \Gamma\oplus_{\ov{\xi}^2_{\gamma_1},\ov{\xi}^1_{\gamma}} \frac{\mathds{Z}}{v\mathds{Z}} & \frac{\mathds{Z}}{v\mathds{Z}} & 0,\\};
\draw[->] (BPcomplex-1-1) -- node[above=1pt,font=\scriptsize] {} (BPcomplex-1-2);
\draw[->] (BPcomplex-1-2) -- node[above=1pt,font=\scriptsize] {$\iota$} (BPcomplex-1-3);
\draw[->] (BPcomplex-1-3) -- node[above=1pt,font=\scriptsize] {$\pi$} (BPcomplex-1-4);
\draw[->] (BPcomplex-1-4) -- node[above=1pt,font=\scriptsize] {} (BPcomplex-1-5);
\end{scope}
\end{tikzpicture}
$$
where $\iota$ and $\pi$ are the evident maps, is a central extension of $(\mathds{Z}/v\mathds{Z};\cdot)$ by $\Gamma$ in the sense of~\cite{LV}*{Def\-inition~5.5}.

\item The extension associated with $(\ov{\xi}_{\gamma}^1,\ov{\xi}_{\gamma_1}^2)$ and $(\ov{\xi}_{\gamma'}^1,\ov{\xi}_{\gamma'_1}^2)$ are equivalent if and only if $\gamma'_1 = \gamma_1$ and $v\gamma'=v\gamma$; and each central extension of $(\mathds{Z}/v\mathds{Z};\cdot)$ by $\Gamma$, is equivalent to one of these.

\end{enumerate}
\end{thmx}

\begin{thmx} Assume that $2<u<v$. For each $\gamma,\gamma_1\in \Gamma$ such that $v\gamma_1=u\gamma$, let $\ov{\xi}^1_{\gamma},\ov{\xi}^2_{\gamma_1,\gamma} \colon \frac{\mathds{Z}} {v\mathds{Z}}\times \frac{\mathds{Z}}{v\mathds{Z}}\to \Gamma$ be the maps defined by
$$
\ov{\xi}^1_{\gamma}(\imath_1,\imath_2)\coloneqq \xi^1_{\gamma}([g^{\imath_1}\ot g^{\imath_2}])\quad\text{and}\quad \ov{\xi}^2_{\gamma_1,\gamma}(t\imath+\jmath,\imath_1)\coloneqq \xi^2_{\gamma_1,\gamma}(g^{t\imath+\jmath}\ot g^{\imath_1}),
$$
where $\xi^1_{\gamma}$ and $\xi^2_{\gamma_1,\gamma}$ are as above of Proposition~\ref{complemento 2}, $0\le\imath<u$ and $0\le\jmath<t$. The following facts hold:

\begin{enumerate}

\item $\Gamma\times  \frac{\mathds{Z}}{v\mathds{Z}}$ is a linear cycle set via
$$
(c,\imath)+(c',\imath') \coloneqq \bigl(c+c'+\ov{\xi}^1_{\gamma}(\imath,\imath'),\imath+\imath'\bigr)\quad\text{and}\quad (c,\imath)\cdot (c',\imath') \coloneqq \bigl(c'+\ov{\xi}^2_{\gamma_1,\gamma}(\imath,\imath'),\imath\cdot \imath'\bigr).
$$
Following \cite{LV} we denoted this linear cycle set by $\Gamma\oplus_{\ov{\xi}^2_{\gamma_1,\gamma},\ov{\xi}^1_{\gamma}} \frac{\mathds{Z}}{v\mathds{Z}}$. Moreover
$$
\begin{tikzpicture}
\begin{scope}[yshift=0cm,xshift=0cm]
\matrix(BPcomplex) [matrix of math nodes, row sep=0em, text height=1.5ex, text
depth=0.25ex, column sep=2.5em, inner sep=0pt, minimum height=5mm, minimum width =6mm]
{0 & \Gamma & \Gamma\oplus_{\ov{\xi}^2_{\gamma_1,\gamma},\ov{\xi}^1_{\gamma}} \frac{\mathds{Z}}{v\mathds{Z}} & \frac{\mathds{Z}}{v\mathds{Z}} & 0,\\};
\draw[->] (BPcomplex-1-1) -- node[above=1pt,font=\scriptsize] {} (BPcomplex-1-2);
\draw[->] (BPcomplex-1-2) -- node[above=1pt,font=\scriptsize] {$\iota$} (BPcomplex-1-3);
\draw[->] (BPcomplex-1-3) -- node[above=1pt,font=\scriptsize] {$\pi$} (BPcomplex-1-4);
\draw[->] (BPcomplex-1-4) -- node[above=1pt,font=\scriptsize] {} (BPcomplex-1-5);
\end{scope}
\end{tikzpicture}
$$
where $\iota$ and $\pi$ are the evident maps, is a central extension of $(\mathds{Z}/v\mathds{Z};\cdot)$ by $\Gamma$ in the sense of~\cite{LV}*{Def\-inition~5.5}.

\item The extension associated with $(\ov{\xi}_{\gamma}^1,\ov{\xi}_{\gamma_1,\gamma}^2)$ and $(\ov{\xi}_{\gamma'}^1,\ov{\xi}_{\gamma'_1,\gamma'}^2)$ are equivalent if and only if $\gamma_1-\gamma'_1\in u \Gamma$ and $t(\gamma_1-\gamma'_1)=\gamma-\gamma'$; and each central extension of $(\mathds{Z}/v\mathds{Z};\cdot)$ by $\Gamma$, is equivalent to one of these.

\end{enumerate}
\end{thmx}

\begin{thmx} Assume that $u=2$ and $v=4$. For each $\gamma,\gamma_1,\gamma'_1\in \Gamma$ such that $4\gamma_1=2\gamma$ and $2\gamma'_2=0$, let $\ov{\xi}^1_{\gamma},\ov{\xi}^2_{\gamma_1,\gamma'_1,\gamma}\colon \frac{\mathds{Z}} {4\mathds{Z}}\times \frac{\mathds{Z}}{v\mathds{Z}}\to \Gamma$ be the maps defined by
$$
\ov{\xi}^1_{\gamma}(\imath_1,\imath_2)\coloneqq \xi^1_{\gamma}([g^{\imath_1}\ot g^{\imath_2}])\quad\text{and}\quad \ov{\xi}^2_{\gamma_1,\gamma'_1,\gamma}(2\imath+\jmath,\imath_1)\coloneqq \xi^2_{\gamma_1,\gamma'_1,\gamma}(g^{2\imath+\jmath}\ot g^{\imath_1}),
$$
where $\xi^1_{\gamma}$ and $\xi^2_{\gamma_1,\gamma'_1,\gamma}$ are as above of Proposition~\ref{complemento 2 p=2}, $0\le\imath,\jmath<2$. The following facts hold:

\begin{enumerate}

\item $\Gamma\times \frac{\mathds{Z}}{v\mathds{Z}}$ is a linear cycle set via
$$
(c,\imath)+(c',\imath') \coloneqq \bigl(c+c'+\ov{\xi}^1_{\gamma}(\imath,\imath'),\imath+\imath'\bigr)\quad\text{and}\quad (c,\imath)\cdot (c',\imath') \coloneqq \bigl(c'+\ov{\xi}^2_{\gamma_1,\gamma'_1,\gamma}(\imath,\imath'),\imath\cdot \imath'\bigr).
$$
Following \cite{LV} we denoted this linear cycle set by $\Gamma\oplus_{\ov{\xi}^2_{\gamma_1,\gamma'_1,\gamma},\ov{\xi}^1_{\gamma}} \frac{\mathds{Z}}{v\mathds{Z}}$. Moreover
$$
\begin{tikzpicture}
\begin{scope}[yshift=0cm,xshift=0cm]
\matrix(BPcomplex) [matrix of math nodes, row sep=0em, text height=1.5ex, text
depth=0.25ex, column sep=2.5em, inner sep=0pt, minimum height=5mm, minimum width =6mm]
{0 & \Gamma & \Gamma\oplus_{\ov{\xi}^2_{\gamma_1,\gamma'_1,\gamma},\ov{\xi}^1_{\gamma}} \frac{\mathds{Z}}{v\mathds{Z}} & \frac{\mathds{Z}}{v\mathds{Z}} & 0,\\};
\draw[->] (BPcomplex-1-1) -- node[above=1pt,font=\scriptsize] {} (BPcomplex-1-2);
\draw[->] (BPcomplex-1-2) -- node[above=1pt,font=\scriptsize] {$\iota$} (BPcomplex-1-3);
\draw[->] (BPcomplex-1-3) -- node[above=1pt,font=\scriptsize] {$\pi$} (BPcomplex-1-4);
\draw[->] (BPcomplex-1-4) -- node[above=1pt,font=\scriptsize] {} (BPcomplex-1-5);
\end{scope}
\end{tikzpicture}
$$
where $\iota$ and $\pi$ are the evident maps, is a central extension of $(\mathds{Z}/v\mathds{Z};\cdot)$ by $\Gamma$ in the sense of~\cite{LV}*{Def\-inition~5.5}.

\item The extension associated with $(\ov{\xi}_{\gamma}^1,\ov{\xi}_{\gamma_1,\gamma'_1,\gamma}^2)$ and $(\ov{\xi}_{\ov{\gamma}}^1,\ov{\xi}_{\ov{\gamma}_1, \ov{\gamma}'_1,\ov{\gamma}}^2)$ are equivalent if and only if $\gamma_1-\ov{\gamma}_1\in 2 \Gamma$, $\gamma-\ov{\gamma}=2(\gamma_1-\ov{\gamma}_1)$ and $\ov{\gamma}'_1=\gamma'_1$.; and each central extension of $(\mathds{Z}/v\mathds{Z};\cdot)$ by $\Gamma$, is equivalent to one of these.

\end{enumerate}
\end{thmx}

In order to prove these results we first compute the normalized full linear cycle set cohomology $\Ho_N^2(\mathcal{A},\Gamma)$. This is done in Theorems~\ref{caso t=1}, \ref{caso 1<t<u} and~\ref{caso t=u=p}.

\section{Preliminaries}\label{section:Preliminaries}

In this paper we work in the category of abelian groups, all the maps are $\mathds{Z}$-linear, $\ot$ means $\ot_{\mathds{Z}}$ and $\Hom$ means $\Hom_{\mathds{Z}}$.

\subsection{Group homology}\label{Hochschild homology of algebras}
Let $G$ be a group, let $D\coloneqq \mathds{Z}[G]$ and let $\ov{D}\coloneqq D/\mathds{Z}1$. We call $\Ss_.(G)$ the simplicial complex of right $D$-modules with objects $\Ss_n(G)\coloneqq D^{\ot n+1}$, face maps $\mu_i\colon \Ss_n(G) \to \Ss_{n-1}(G)$ and degeneracy maps $\ep_i\colon\Ss_n(G) \to \Ss_{n+1}(G)$ ($i=0,\dots,n$), defined by:
\begin{align*}
&\mu_i(x_1\ot\cdots \ot x_{n+1}) \coloneqq \begin{cases} x_2\ot\cdots \ot x_{n+1} &\text{if $i=0$,}\\
x_1\ot \cdots \ot x_ix_{i+1}\ot\cdots \ot x_{n+1} &\text{if $0<i\le n$,} \end{cases}\\
&\epsilon_i(x_1\ot\cdots \ot x_{n+1}) \coloneqq x_1\ot \cdots \ot x_{\imath}\ot 1\ot x_{\imath+1}\ot\cdots \ot x_{n+1}.
\end{align*}
The chain complex associated with $\Ss_.(G)$ is the {\em bar resolution} $(D^{\ot *+1},b'_*)$ of the trivial right $D$-module $\mathds{Z}$, and the chain complex $(\ov{D}^{\ot *}\ot D,b'_*)$, obtained dividing  $(D^{\ot *+1},b'_*)$ by the subcomplex generated by images of the degeneracy maps is the {\em bar normalized resolution} of $\mathds{Z}$.

\smallskip

Let $\Upsilon$ be the family of all the epimorphism of right $D$-modules which split as morphisms of abelian groups. We say that a right $D$-module $X$ is {\em $\Upsilon$-relative projective} if for each $f\colon Y_1\to Y_2$ in $\Upsilon$ and each right $D$-module map $g\colon X\to Y_2$, there exists a right $D$-module map $h\colon X\to Y_1$ such that $g=f\xcirc h$. It is well known that a right $D$-module $X$ is $\Upsilon$-relative projective if and only if there exists an abelian group $X'$ such that $X$ is a direct sum of $X'\ot D$. A complex of right $D$-modules $(X_*,d_*)$ is a {\em $\Upsilon$-relative projective resolution of $\mathds{Z}$} if each $X_n$ is $\Upsilon$-relative projective, and there exists a right $D$-module morphism $\pi\colon X_0 \to \mathds{Z}$ such that
$$
\begin{tikzpicture}
\begin{scope}[yshift=0cm,xshift=0cm]
\matrix(BPcomplex) [matrix of math nodes, row sep=0em, text height=1.5ex, text
depth=0.25ex, column sep=2.5em, inner sep=0pt, minimum height=5mm, minimum width =6mm]
{\mathds{Z} & X_0 & X_1 & X_2 & X_3 & X_4 &\cdots,\\};
\draw[<-] (BPcomplex-1-1) -- node[above=1pt,font=\scriptsize] {$\pi$} (BPcomplex-1-2);
\draw[<-] (BPcomplex-1-2) -- node[above=1pt,font=\scriptsize] {$d_1$} (BPcomplex-1-3);
\draw[<-] (BPcomplex-1-3) -- node[above=1pt,font=\scriptsize] {$d_2$} (BPcomplex-1-4);
\draw[<-] (BPcomplex-1-4) -- node[above=1pt,font=\scriptsize] {$d_3$} (BPcomplex-1-5);
\draw[<-] (BPcomplex-1-5) -- node[above=1pt,font=\scriptsize] {$d_4$} (BPcomplex-1-6);
\draw[<-] (BPcomplex-1-6) -- node[above=1pt,font=\scriptsize] {$d_5$} (BPcomplex-1-7);
\end{scope}
\end{tikzpicture}
$$
is contractil as complex of abelian groups. The complex $(\ov{D}^{\ot *}\ot D,b'_*)$ is an $\Upsilon$-relative projective~res\-o\-lu\-tion of $\mathds{Z}$. Let $\pi\colon D\to \mathds{Z}$ be the augmentation map. A contracting homotopy of
$$
\begin{tikzpicture}
\begin{scope}[yshift=0cm,xshift=0cm]
\matrix(BPcomplex) [matrix of math nodes,row sep=0em, column sep=2.5em, text height=1.5ex, text
depth=0.25ex , inner sep=3pt, minimum height=5mm, minimum width =6mm]
{\mathds{Z} & D & \ov{D}\ot D& \ov{D}^{\ot 2}\ot D &\ov{D}^{\ot 3}\ot D &\cdots,\\};
\draw[<-] (BPcomplex-1-1) -- node[above=1pt,font=\scriptsize] {$\pi$} (BPcomplex-1-2);
\draw[<-] (BPcomplex-1-2) -- node[above=1pt,font=\scriptsize] {$b'_1$} (BPcomplex-1-3);
\draw[<-] (BPcomplex-1-3) -- node[above=1pt,font=\scriptsize] {$b'_2$} (BPcomplex-1-4);
\draw[<-] (BPcomplex-1-4) -- node[above=1pt,font=\scriptsize] {$b'_3$} (BPcomplex-1-5);
\draw[<-] (BPcomplex-1-5) -- node[above=1pt,font=\scriptsize] {$b'_4$} (BPcomplex-1-6);
\end{scope}
\end{tikzpicture}
$$
as a complex of abelian groups, is the degree~$1$ map~$\xi_*$, given by $\xi_{n+1}(\bx)\coloneqq (-1)^{n+1}\bx\ot 1$ for $\bx\in \ov{D}^{\ot n}\ot D$. Using relative projective resolutions, a theory of relative derived functors can be developed, which is similar to the standard one (see \cite{H}). Thus, we can define the group homology of $G$ with coefficients in a left $D$-module $M$ as the $\Tor$ relative to the family of epimorphisms $\Upsilon$. Consequently, the {\em group homology $\Ho_*(G,M)$, of $G$ with coefficients in $M$}, is the homology of $(D\ot \ov{D}^{\ot *},b'_*)\ot_D M$. There are canonical~iden\-ti\-fi\-ca\-tions $\ov{\digamma}_{\!n}\colon \ov{D}^{\ot n}\ot M \to \bigl(\ov{D}^{\ot n}\ot D\bigr)\ot_D M$, given by $\ov{\digamma}_{\!n}(\bx\ot m)\coloneqq (\bx\ot 1)\ot_D m$. Using them we obtain that $\bigl(\ov{D}^{\ot *}\ot D,b'_*\bigr)\ot_D M\simeq \bigl(\ov{D}^{\ot *}\ot M,b_*\bigr)$, where
\begin{align*}
b_n(x_1\ot\cdots\ot x_n\ot m)&\coloneqq x_2\ot \cdots\ot x_n\\
& + \sum_{i=1}^{n-1} (-1)^{\imath} x_1\ot \cdots \ot x_{\imath-1}\ot x_{\imath}x_{\imath+1}\ot x_{\imath+2}\ot\cdots\ot x_n\ot m\\
& + (-1)^n x_1\ot\cdots\ot x_{n-1}\ot x_nm.
\end{align*}
We call $\cramped{\bigl(\ov{D}^{\ot *}\ot M,b_*\bigr)}$ the canonical normalized complex of $G$ with coefficients in $M$. Since the right~$D$-mo\-dules $D^{\ot n+1}$ and $\ov{D}^{\ot n}\ot D$ are projective, the group homology can be defined using the usual functor~$\Tor$. The main purpose of our comment on relative derived functors above is to make subsection~\ref{una resolucion} and the reference~\cite{GG} more understandable.

\subsection{Linear cycle sets}
A linear cycle set $\mathcal{A} \coloneqq (A;\cdot)$ is an abelian additive group $A$, endowed with a binary operation $\cdot$ such that the left translations $a\mapsto a'\cdot a$ are permutations of $A$ and the following conditions are fulfilled
\begin{align}
&(a\cdot a')\cdot (a\cdot a'') = (a'\cdot a)\cdot (a'\cdot a''),\label{condicion1}\\
&a\cdot (a'+a'') = a\cdot a' +  a\cdot a'',\label{condicion2}\\
& (a+a')\cdot a'' = (a\cdot a')\cdot (a\cdot a'').\label{condicion3}
\end{align}
We will use multiplicative notation. Let $G_A\coloneqq \{X^a:a\in A\}$, endowed with the group structure given by $X^aX^{a'} \coloneqq X^{a+a'}$. We set $X^a\cdot X^{a'}\coloneqq X^{a\cdot a'}$ and $X^a\cdot a'\coloneqq a\cdot a'$.

\subsubsection{The normalized full linear cycle set (co)homology}\label{The full linear cycle set (co)homology}
In \cite{LV}*{Section~4} the authors introduce theories of (co)homology, $\Ho^N_*(\mathcal{A},\Gamma)$ and $\Ho_N^*(\mathcal{A},\Gamma)$, that we recall now. For each $s\ge 1$, we let $\sh(\ov{D}^{\ot s})$ denote the subgroup of $\ov{D}^{\ot s}$ generated by the shuffles
$$
\sum_{\sigma\in \sh_{l,s-l}} \sg(\sigma) d_{\sigma^{-1}(1)}\ot\cdots\ot d_{\sigma^{-1}(s)},
$$
taken for all $1\le l < s$ and $d_k\in\ov{D}$. Here $\sh_{l,s-l}$ is the subset of all the permutations $\sigma$ of $s$ elements satisfying $\sigma(1)<\cdots<\sigma(l)$ and $\sigma(l+1)<\cdots<\sigma(s)$. For each $r\ge 0$ and $s\ge 1$, let $\wh{C}^N_{rs}(\mathcal{A},\mathds{Z})\coloneqq \ov{D}^{\ot r}\ot \ov{M}(s)$, where $\ov{M}(s)\coloneqq \frac{\ov{D}^{\ot s}}{\sh(\ov{D}^{\ot s})}$. Given $g_1,\dots,g_s\in G_A$, we let $[g_1\ot\cdots\ot g_s]$ denote the class of $g_1\ot\cdots\ot g_s$ in $\ov{M}(s)$. Consider the double complex $\bigl(\wh{C}^N_{**}(\mathcal{A},\mathds{Z}),\partial^{\mathrm{h}}_{**},\partial^{\mathrm{v}}_{**}\bigr)$, where
\begin{align*}
& \partial^{\mathrm{h}}_{rs}(g_1\ot\cdots \ot g_r \ot [g_{r+1} \ot\cdots \ot g_{r+s}]) \coloneqq g_1\cdot g_2\ot \cdots\ot g_1\cdot g_r\ot [g_1\cdot g_2\ot\cdots\ot g_1\cdot g_{r+s}]\\
&\phantom{\partial^{\mathrm{h}}_{rs}(g_1\ot\cdots \ot g_r)} + \sum_{\jmath=1}^{r-1} (-1)^{\jmath} g_1\ot\cdots\ot g_{\jmath-1}\ot g_{\jmath}g_{\jmath+1}\ot g_{\jmath+2}\ot g_r\ot [g_{r+1}\ot\cdots\ot g_{r+s}]\\
&\phantom{\partial^{\mathrm{h}}_{rs}(g_1\ot\cdots \ot g_r)} + (-1)^r g_1\ot\cdots\ot g_{r-1}\ot [g_{r+1}\ot\cdots \ot g_{r+s}]\\
\shortintertext{and}
& \partial^{\mathrm{v}}_{rs}(g_1\ot\cdots \ot g_r \ot [g_{r+1} \ot\cdots \ot g_{r+s}]) \coloneqq (-1)^{r+1} g_1\ot\cdots \ot g_r\ot [g_{r+2} \ot\cdots \ot g_{r+s}]\\
&\phantom{\partial^{\mathrm{h}}_{rs}(g_1\ot\cdots \ot g_r)} + \sum_{\jmath=r+1}^{r+s-1} (-1)^{\jmath+1} g_1\ot\cdots\ot g_r\ot [g_{r+1}\ot\cdots\ot g_{\jmath-1}\ot g_{\jmath}g_{\jmath+1}\ot g_{\jmath+2}\ot\cdots\ot g_{r+s}]\\
&\phantom{\partial^{\mathrm{h}}_{rs}(g_1\ot\cdots \ot g_r)} + (-1)^{r+s+1} g_1\ot\cdots\ot g_r\ot [g_{r+1}\ot\cdots\ot g_{r+s-1}].
\end{align*}
Recall that the total complex of $\bigl(\wh{C}^N_{**}(\mathcal{A},\mathds{Z}),\partial^{\mathrm{h}}_{**},\partial^{\mathrm{v}}_{**}\bigr)$ is the chain complex $\bigl(\wh{C}^N_*(\mathcal{A},\mathds{Z}),\partial_*\bigr)$, where
$$
\wh{C}^N_n(\mathcal{A},\mathds{Z})\coloneqq \bigoplus_{r+s=n} \wh{C}^N_{rs}(\mathcal{A},\mathds{Z})\quad\text{and}\quad \partial_n\vert_{\wh{C}^N_{rs}(\mathcal{A},\mathds{Z})} \coloneqq \partial^{\mathrm{h}}_{rs}+\partial^{\mathrm{v}}_{rs}.
$$
Let $\Gamma$ be an abelian additive group. The {\em normalized full homology groups} and the {\em normalized full cohomology groups} of $\mathcal{A}$ with coefficients in $\Gamma$ are the homology groups of $\wh{C}^N_*(\mathcal{A},\Gamma)\coloneqq \Gamma\ot (\wh{C}^N_*(\mathcal{A},\mathds{Z}),\partial_*)$ and the cohomology groups of $\wh{C}_N^*(\mathcal{A},\Gamma)\coloneqq \Hom\bigl((\wh{C}^N_*(\mathcal{A},\mathds{Z}),\partial_*),\Gamma\bigr)$, respectively. We let $\wh{\Ho}^N_*(\mathcal{A},\Gamma)$ and $\wh{\Ho}_N^*(\mathcal{A},\Gamma)$ denote the full normalized homology and the full normalized cohomology, of $\mathcal{A}$ with coefficients in $\Gamma$.

\begin{remark}\label{son iguales2} The complex $\wh{C}^N_*(\mathcal{A},\Gamma)$ is not the complex $\bigl(C^N_*(\mathcal{A},\Gamma),\partial_*\bigr)$ introduced in \cite{LV}*{Def\-i\-nition~4.2}, but they are isomorphic via the maps $\Xi_{rs}\colon \wh{C}^N_{rs}(\mathcal{A},\Gamma)\to C^N_{rs}(\mathcal{A},\Gamma)$, given by
$$
\Xi_{rs}(X^{a_1}\ot\cdots\ot X^{a_s}\ot [X^{a_{s+1}}\ot\cdots \ot X^{a_{r+s}}])\coloneqq (a_1,\dots,a_s,a_{s+1},\dots,a_{r+s}).
$$
Similarly, $\wh{C}_N^*(\mathcal{A},\Gamma)\simeq \bigl(C_N^*(\mathcal{A},\Gamma), \partial^*\bigr)$, and so $\wh{\Ho}^N_*(\mathcal{A},\Gamma) = \Ho^N_*(\mathcal{A},\Gamma)$ and $\wh{\Ho}_N^*(\mathcal{A},\Gamma) = \Ho_N^*(\mathcal{A},\Gamma)$.
\end{remark}

\subsection{The perturbation lemma}
Next, we recall the perturbation lemma. We present the version given in~\cite{C}.

\smallskip

A {\em special deformation retract}
\begin{equation}\label{(a)}
\begin{tikzpicture}[baseline=(current bounding box.center)]
\draw (-0.1,0) node[] {$(X_*,d_*)$};\draw (2.6,0) node {$(C_*,\partial_*)$};\draw[<-] (0.6,0.12) -- node[above=-2pt,font=\scriptsize] {$p_*$}(1.9,0.12);\draw[->] (0.6,-0.12) -- node[below=-2pt,font=\scriptsize] {$\imath_*$}(1.9,-0.12);\draw (4.2,0) node {$C_*$};\draw (6.7,0) node {$C_{*+1}$,};\draw[->] (4.4,0) -- node[above=-2pt,font=\scriptsize] {$h_{*+1}$}(6.1,0);
\end{tikzpicture}
\end{equation}
consists of the following:

\begin{enumerate}

\smallskip

\item Chain complexes $(X,d)$, $(C,\partial)$ and morphisms $\imath$, $p$ between them, such that $p\xcirc \imath = \ide$.

\smallskip

\item A homotopy $h$ from $\imath\xcirc p$ to $\ide$, such that $h\xcirc \imath = 0$, $p\xcirc h = 0$ and $h\xcirc h = 0$.
\end{enumerate}

\smallskip

A {\em perturbation} of~\eqref{(a)} is a map $\delta_*\colon C_*\to C_{*-1}$ such that $(\partial+\delta)^2 = 0$. We call it {\em small} if $\ide -\delta\xcirc h$ is invertible. In this case we write $A \coloneqq  (\ide -\delta\xcirc h)^{-1}\xcirc\delta$ and we consider the diagram
\begin{equation}\label{(b)}
\begin{tikzpicture}[baseline=(current bounding box.center)]
\draw (-0.1,0) node[] {$(X_*,d_*^1)$};\draw (3,0) node {$(C_*,\partial_*+\delta_*)$};\draw[<-] (0.6,0.12) -- node[above=-2pt,font=\scriptsize] {$p_*^1$}(1.9,0.12);\draw[->] (0.6,-0.12) -- node[below=-2pt,font=\scriptsize] {$\imath_*^1$}(1.9,-0.12);\draw (5.1,0) node {$C_*$};\draw (7.3,0) node {$C_{*+1}$,};\draw[->] (5.3,0) -- node[above=-2pt,font=\scriptsize] {$h_{*+1}^1$}(6.8,0);
\end{tikzpicture}
\end{equation}
where $d^1\coloneqq d + p\xcirc A\xcirc i$, $i^1\coloneqq  i + h\xcirc A\xcirc i$, $p^1\coloneqq  p + p\xcirc A\xcirc h$ and $h^1\coloneqq  h + h\xcirc A\xcirc h$.

\smallskip

In all the cases considered in this paper the morphism $\delta\xcirc h$ is locally nilpotent (in other words, for all $x\in C_*$ there exists $n\in \mathds{N}$ such that $(\delta\xcirc h)^n(x)=0$). Consequently, $(\ide -\delta\xcirc h)^{-1} =\sum_{n=0}^{\infty} (\delta\xcirc h)^n$.

\begin{theorem}[{\cite{C}}]\label{lema de perturbacion} If $\delta$ is a small perturbation of~\eqref{(a)}, then the diagram~\eqref{(b)} is a special deformation retract.
\end{theorem}

\begin{proposition}\label{para complejos dobles} Consider morphisms of double complexes
\begin{equation}\label{(ad1)}
\begin{tikzpicture}[baseline=(current bounding box.center)]
\draw (-0.5,0) node[] {$(X_{**},d^{\mathrm{h}}_{**},d^{\mathrm{v}}_{**})$};\draw (3.2,0) node {$(C_{**},\partial^{\mathrm{h}}_{**},\partial^{\mathrm{v}}_{**})$,};\draw[<-] (0.6,0.12) -- node[above=-2pt,font=\scriptsize] {$p_{**}$}(1.9,0.12);\draw[->] (0.6,-0.12) -- node[below=-2pt,font=\scriptsize] {$\imath_{**}$}(1.9,-0.12);
\end{tikzpicture}
\end{equation}
such that $p_{**}\xcirc \imath_{**} = \ide$. Assume that in each row $s$ we have a special~de\-for\-ma\-tion retract
\begin{equation}\label{(aa1)}
\begin{tikzpicture}[baseline=(current bounding box.center)]
\draw (-0.2,0) node[] {$(X_{*s},d^{\mathrm{h}}_{*s})$};\draw (2.8,0) node {$(C_{*s},\partial^{\mathrm{h}}_{*s})$};\draw[<-] (0.6,0.12) -- node[above=-2pt,font=\scriptsize] {$p_{*s}$}(1.9,0.12);\draw[->] (0.6,-0.12) -- node[below=-2pt,font=\scriptsize] {$\imath_{*s}$}(1.9,-0.12);\draw (4.6,0) node {$C_{*s}$};\draw (7.1,0) node {$C_{*+1,s}$,};\draw[->] (5,0) -- node[above=-2pt,font=\scriptsize] {$h_{*+1,s}$}(6.5,0);
\end{tikzpicture}
\end{equation}
endowed with a small perturbation $\delta^{\mathrm{h}}_{*s}\colon C_{*s}\to C_{*-1,s}$. Let $A_{**} \coloneqq  (\ide -\delta^{\mathrm{h}}_{**}\xcirc h_{**})^{-1}\xcirc\delta^{\mathrm{h}}_{**}$ and consider the di\-a\-gram
\begin{equation}\label{(ad2)}
\begin{tikzpicture}[baseline=(current bounding box.center)]
\draw (-0.5,0) node[] {$(X_{**},d^{h1}_{**},d^{\mathrm{v}}_{**})$};\draw (3.5,0) node {$(C_{**},\partial^{\mathrm{h}}_{**}+\delta^{\mathrm{h}}_{**},\partial^{\mathrm{v}}_{**})$};\draw[<-] (0.6,0.12) -- node[above=-2pt,font=\scriptsize] {$p^1_{**}$}(1.9,0.12);\draw[->] (0.6,-0.12) -- node[below=-2pt,font=\scriptsize] {$\imath^1_{**}$}(1.9,-0.12); \draw (6.2,0) node {$C_{**}$};\draw (8.7,0) node {$C_{*+1,*}$,};\draw[->] (6.5,0) -- node[above=-2pt,font=\scriptsize] {$h^1_{*+1,*}$,}(8,0);
\end{tikzpicture}
\end{equation}
where $d^{h1}\coloneqq d^{\mathrm{h}} + p\xcirc A\xcirc \imath$, $\imath^1\coloneqq  \imath + h\xcirc A\xcirc \imath$, $p^1\coloneqq  p + p\xcirc A\xcirc h$ and $h^1\coloneqq  h + h\xcirc A\xcirc h$. The following facts hold:

\begin{enumerate}

\smallskip

\item The maps $\imath^1_{**}$ and $p^1_{**}$ are morphisms of double complexes such that $p^1_{**} \xcirc \imath^1_{**} = \ide$.

\smallskip

\item For each row $s$, the map $h^1_{*+1,s}$ is a homotopy from $\imath^1_{*s}\xcirc p^1_{*s}$ to $\ide$.
\end{enumerate}
\end{proposition}

\begin{proof} Let $(X_*,d^t_*)$ and $(C_*,\partial^t_*)$ be the total chain complexes of $(X_{**},d^{\mathrm{h}}_{**},d^{\mathrm{v}}_{**})$ and $(C_{**},\partial^{\mathrm{h}}_{**},\partial^{\mathrm{v}}_{**})$, respectively. We have an homotopy equivalence data
\begin{equation}\label{(aa2)}
\begin{tikzpicture}[baseline=(current bounding box.center)]
\draw (-0.1,0) node[] {$(X_*,d^t_*)$};\draw (2.6,0) node {$(C_*,\partial^t_*)$};\draw[<-] (0.6,0.12) -- node[above=-2pt,font=\scriptsize] {$p^t_*$}(1.9,0.12);\draw[->] (0.6,-0.12) -- node[below=-2pt,font=\scriptsize] {$\imath^t_*$}(1.9,-0.12);\draw (4.2,0) node {$C_*$};\draw (6.7,0) node {$C_{*+1}$,};\draw[->] (4.4,0) -- node[above=-2pt,font=\scriptsize] {$h^t_{*+1}$}(6.1,0);
\end{tikzpicture}
\end{equation}
where $\imath^t_*$, $p_*^t$ an $h_*^t$ are given by $\imath^t_n \coloneqq \bigoplus_{r+s = n} \imath_{rs}$, $p^t_n \coloneqq \bigoplus_{r+s = n} p_{rs}$ and $h^t_{n+1} \coloneqq \bigoplus_{r+s=n+1} h_{rs}$. Consider the small perturbation $\delta^t_*\colon C_*\to C_{*-1}$, given by $\delta^t_n\coloneqq \bigoplus_{r+s = n} \delta^h_{rs}$. The result follows immediately by applying the perturbation lemma to this case.
\end{proof}

\section{A complex for the group homology of cyclic groups}
Let $C_v$ be a cyclic group of order $v\in \mathds{N}$ and let $D\coloneqq \mathds{Z}[C_v]$. In this section we construct a chain complex suitable for our purposes, giving the group homology of $C_v$ with coefficients in an abelian group~$M$, considered as a left $D$-module via the trivial action. This complex is the complex $(X_*,d_*)$ in a special deformation retract as in~\eqref{(a)}, in which $(C_*,\partial_*)$ is the normalized bar complex of $C_v$ with coe\-fficients in $M$. It is natural to try to use the minimal resolution of $C_v$ in order to construct $(X_*,d_*)$, but this does not work because, in this case, the perturbation is not small. So we are forced to use a more involved complex.

\subsection{A resolution for a cyclic group}\label{una resolucion}
Let $v,u,t\in \mathds{N}$ such that $u>1$ and $ut=v$. Consider the cycle groups $C_v\coloneqq \langle g \rangle$, $C_u\coloneqq \langle x \rangle$ and $C_t\coloneqq \langle y \rangle$ of order $v$, $u$ and $t$, respectively. The group $C_v$ is isomorphic to the crossed product $C_u \rtimes_{\zeta}C_t$, in which $C_t$ acts trivially on $C_u$ and $\zeta$ is the cocycle given by
$$
\zeta(y^{\jmath},y^{\jmath'})\coloneqq \begin{cases} 1 &\text{if $\jmath+\jmath'<t$},\\ x &\text{otherwise,}\end{cases}
$$
where $0\le \jmath,\jmath'<t$. We recall that $C_u \rtimes_{\zeta}C_t=\{x^{\imath}w_{y^{\jmath}}: 0\le \imath < u \text{ and } 0\le \jmath<t\}$ endowed with~the~multpli\-ca\-tion map
$$
x^{\imath}w_{y^{\jmath}} \, x^{\imath'}w_{y^{\jmath'}}= x^{\imath+\imath'}\zeta(y^{\jmath},y^{\jmath'})w_{y^{\jmath+\jmath'}}\quad\text{ where $0\le \jmath,\jmath'<t$.}
$$
The map $f\colon C_u \rtimes_{\zeta}C_t \to C_v$, defined by $f(x^{\imath}w_{y^{\jmath}})\coloneqq g^{t\imath+\jmath}$, where $0\le \jmath<t$, is an group isomor\-phism.

\smallskip

Let $E\coloneqq \mathds{Z}[C_u \rtimes_{\zeta}C_t]$. For all $\alpha,\beta\ge 0$, let $Y_{\beta}\coloneqq \mathds{Z}[C_t]$ and $X_{\alpha\beta}\coloneqq E$. The groups $X_{\alpha\beta}$ are right~$E$-mo\-dules via the right regular action and the groups $Y_{\beta}$ are right $E$-modules via $y^lx^{\imath} w_{y^{\jmath}}\coloneqq  y^{\jmath+l}$. Consider the diagram of right $E$-modules and right $E$-module maps
$$
\begin{tikzpicture}
\begin{scope}[yshift=0cm,xshift=0cm]
\matrix(BPcomplex) [matrix of math nodes,row sep=2.5em, text height=1.5ex, text
depth=0.25ex, column sep=2.5em, inner sep=0pt, minimum height=5mm,minimum width =8mm]
{\vdots\\
Y_2 & X_{02} & X_{12} &\cdots\\
Y_1 & X_{01} & X_{11} &\cdots\\
Y_0 & X_{00} & X_{10} &\cdots,\\};
\draw[->] (BPcomplex-1-1) -- node[right=1pt,font=\scriptsize] {$\partial_3$} (BPcomplex-2-1);
\draw[->] (BPcomplex-2-1) -- node[right=1pt,font=\scriptsize] {$\partial_2$} (BPcomplex-3-1);
\draw[->] (BPcomplex-3-1) -- node[right=1pt, font=\scriptsize] {$\partial_1$} (BPcomplex-4-1);
\draw[<-] (BPcomplex-2-1) -- node[above=1pt,font=\scriptsize] {$\upsilon_2$} (BPcomplex-2-2);
\draw[<-] (BPcomplex-2-2) -- node[above=1pt,font=\scriptsize] {$d^0_{12}$} (BPcomplex-2-3);
\draw[<-] (BPcomplex-2-3) -- node[above=1pt,font=\scriptsize] {$d^0_{22}$} (BPcomplex-2-4);
\draw[<-] (BPcomplex-3-1) -- node[above=1pt,font=\scriptsize] {$\upsilon_1$} (BPcomplex-3-2);
\draw[<-] (BPcomplex-3-2) -- node[above=1pt,font=\scriptsize] {$d^0_{11}$} (BPcomplex-3-3);
\draw[<-] (BPcomplex-3-3) -- node[above=1pt,font=\scriptsize] {$d^0_{21}$} (BPcomplex-3-4);
\draw[<-] (BPcomplex-4-1) -- node[above=1pt,font=\scriptsize] {$\upsilon_0$} (BPcomplex-4-2);
\draw[<-] (BPcomplex-4-2) -- node[above=1pt,font=\scriptsize] {$d^0_{10}$} (BPcomplex-4-3);
\draw[<-] (BPcomplex-4-3) -- node[above=1pt,font=\scriptsize] {$d^0_{20}$} (BPcomplex-4-4);
\end{scope}
\end{tikzpicture}
$$
where $v_{\beta}(w_1)\coloneqq 1$ and
$$
\partial_{2\beta-1}(1)\coloneqq y-1,\quad \partial_{2\beta}(1)\coloneqq \sum_{l=0}^{t-1}y^l,\quad d_{2\alpha-1,\beta}^0(w_1)\coloneqq xw_1 - w_1\quad\text{and}\quad d_{2\alpha,\beta}^0(w_1)\coloneqq \sum_{l=0}^{u-1}x^lw_1.
$$
Clearly, the column and the rows of this diagram are chain complexes.

\begin{proposition}\label{las filas son contractiles} Each one of the rows of the above diagram is contractible as a complex of abelian groups. A contracting homotopy $\sigma^0_{0\beta}\colon Y_{\beta}\to X_{0\beta}$ and $\sigma^0_{\alpha+1,\beta}\colon X_{\alpha\beta}\to X_{\alpha+1,\beta}$ for $\alpha\ge 0$, of the $\beta$-th row, is given by
$$
\sigma^0_{0\beta}(y^{\jmath})\coloneqq w_{y^{\jmath}},\quad \sigma^0_{2\alpha-1,\beta}(x^{\imath}w_{y^{\jmath}})\coloneqq \sum_{l=0}^{\imath-1} x^l w_{y^{\jmath}}\quad\text{and}\quad \sigma^0_{2\alpha,\beta}(x^{\imath}w_{y^{\jmath}}) \coloneqq \delta_{\imath,u-1} w_{y^{\jmath}}
$$
where $0\le\imath<u$ and $\delta_{\imath,u-1}$ is the delta of Kronecker.
\end{proposition}

\begin{proof} We must check that
\begin{equation}\label{a ver}
v_{\beta}\xcirc \sigma^0_{0\beta} = \ide_{Y_{\beta}},\quad \sigma^0_{0\beta}\xcirc v_{\beta} + d^0_{1\beta}\xcirc \sigma^0_{1\beta} = \ide_{X_{0\beta}}\quad\text{and}\quad \sigma^0_{\alpha\beta}\xcirc d^0_{\alpha\beta} + d^0_{\alpha+1,\beta}\xcirc \sigma^0_{\alpha+1,\beta} = \ide_{X_{\alpha\beta}}.
\end{equation}
A direct computation shows that
\begin{align*}
& v_{\beta}\xcirc\sigma^0_{0\beta}(y^{\jmath}) = v_{\beta}\bigl(w_{y^{\jmath}}\bigr) = y^{\jmath},\\
& \sigma^0_{0\beta}\xcirc v_{\beta}(x^{\imath}w_{y^{\jmath}})= \sigma^0_{0\beta}\bigl(y^{\jmath}\bigr) = w_{y^{\jmath}},\\
& d^0_{2\alpha+1,\beta}\xcirc\sigma^0_{2\alpha+1,\beta}(x^{\imath} w_{y^{\jmath}}) = \sum_{l=0}^{\imath-1} d^0_{2\alpha+1,\beta}\bigl(x^lw_{y^{\jmath}}\bigr) = x^{\imath}w_{y^{\jmath}} - w_{y^{\jmath}},\\
& d^0_{2\alpha,\beta}\xcirc\sigma^0_{2\alpha,\beta}(x^{\imath} w_{y^{\jmath}}) = \delta_{\imath,u-1} d^0_{2\alpha,\beta} (w_{y^j}) = \delta_{\imath,u-1}\sum_{l=0}^{u-1} x^lw_{y^j}\\
& \sigma^0_{2\alpha+1,\beta}\xcirc d^0_{2\alpha+1,\beta}(x^{\imath} w_{y^{\jmath}}) = \sigma^0_{2\alpha+1,\beta}(x^{\imath+1} w_{y^{\jmath}}-x^{\imath} w_{y^{\jmath}}) = \begin{cases} x^{\imath}w_{y^{\jmath}} &\text{if $0\le \imath<u-1$,}\\ -\sum_{l=0}^{u-2} x^l w_{y^{\jmath}} & \text{if $\imath=u-1$,} \end{cases}\\
& \sigma^0_{2\alpha,\beta}\xcirc d^0_{2\alpha,\beta}(x^{\imath} w_{y^{\jmath}}) = \sum_{l=0}^{u-1} \sigma^0_{2\alpha,\beta}\bigl(x^{l+\imath} w_{y^{\jmath}}\bigr) = w_{y^{\jmath}}.
\end{align*}
Equalities~\eqref{a ver} follows immediately from these facts.
\end{proof}	

\begin{proposition} Consider $\mathds{Z}$ as a right $E$-module via the trivial action. The complex of right $E$-modules
$$
\begin{tikzpicture}
\begin{scope}[yshift=0cm,xshift=0cm]
\matrix(BPcomplex) [matrix of math nodes,row sep=0em, text height=1.5ex, text
depth=0.25ex, column sep=2.5em, inner sep=0pt, minimum height=5mm,minimum width =6mm]
{\mathds{Z} & Y_0 & Y_1 & Y_2 & Y_3 & Y_4 & Y_5 &\cdots,\\};
\draw[<-] (BPcomplex-1-1) -- node[above=1pt,font=\scriptsize] {$\pi$} (BPcomplex-1-2);
\draw[<-] (BPcomplex-1-2) -- node[above=1pt,font=\scriptsize] {$\partial_1$} (BPcomplex-1-3);
\draw[<-] (BPcomplex-1-3) -- node[above=1pt,font=\scriptsize] {$\partial_2$} (BPcomplex-1-4);
\draw[<-] (BPcomplex-1-4) -- node[above=1pt,font=\scriptsize] {$\partial_3$} (BPcomplex-1-5);
\draw[<-] (BPcomplex-1-5) -- node[above=1pt,font=\scriptsize] {$\partial_4$} (BPcomplex-1-6);
\draw[<-] (BPcomplex-1-6) -- node[above=1pt,font=\scriptsize] {$\partial_5$} (BPcomplex-1-7);
\draw[<-] (BPcomplex-1-7) -- node[above=1pt,font=\scriptsize] {$\partial_6$} (BPcomplex-1-8);
\end{scope}
\end{tikzpicture}
$$
where $\pi$ is the right $E$-module morphism given by $\pi(w_1)\coloneqq 1$, is contractible as a complex of abelian groups. A contracting homotopy $\sigma_0^{-1}\colon E\to Y_0$ and $\sigma^{-1}_{\beta+1}\colon Y_{\beta}\to Y_{\beta+1}$ for $\beta\ge 0$, is given by
$$
\sigma^{-1}_0(1)\coloneqq w_1,\quad \sigma^{-1}_{2\beta}(y^{\jmath})\coloneqq \delta_{t-1,\jmath}\quad\text{and}\quad \sigma^{-1}_{2\beta+1}(y^{\jmath})\coloneqq \sum_{l=0}^{\jmath-1} y^l
$$
where $0\le \jmath<t$.
\end{proposition}

\begin{proof} A direct computation shows that
\begin{align*}
& \pi\xcirc\sigma^{-1}_0(1) = \pi(1) = 1,\\
& \sigma^{-1}_0\xcirc \pi(y^{\jmath}) = \sigma^{-1}_0(1) = 1,\\
& \partial_{2\beta+1}\xcirc\sigma^{-1}_{2\beta+1}(y^{\jmath}) = \sum_{l=0}^{\jmath-1} \partial_{2\beta+1}(y^l) = y^{\jmath} - 1,\\
& \partial_{2\beta}\xcirc\sigma^{-1}_{2\beta}(y^{\jmath}) = \partial_{2\beta}(\delta_{t-1,\jmath}) = \delta_{t-1,\jmath} \sum_{l=0}^{t-1} y^l,\\
& \sigma^{-1}_{2\beta+1}\xcirc\partial_{2\beta+1}(y^{\jmath}) = \sigma^{-1}_{2\beta+1}\bigl(y^{\jmath+1}- y^{\jmath}\bigr) = \begin{cases} y^{\jmath} &\text{if $\jmath<t-1$,}\\ -\sum_{l=0}^{t-2} y^l &\text{if $\jmath=t-1$,} \end{cases}\\
& \sigma^{-1}_{2\beta}\xcirc\partial_{2\beta}(y^{\jmath}) = \sum_{l=0}^{t-1} \sigma^{-1}_{2\beta}(y^l) = 1.
\end{align*}
The result follows easily from these facts.
\end{proof}

For $\alpha\ge 0$ and $1\le l\le \beta$, we define right $E$-module maps $d^l_{\alpha\beta}\colon X_{\alpha\beta}\to X_{\alpha+l-1,\beta-l}$, \index{db@$d^l_{\alpha\beta}$|dotfillboldidx} recursively by:
\begin{equation}\label{def de dl}
d^l(w_1)\coloneqq \begin{cases}
-\sigma^0\xcirc\partial\xcirc\upsilon(w_1) &\text{if $l=1$ and $\alpha=0$,}\\
-\sigma^0\xcirc d^1\xcirc d^0(w_1) &\text{if $l=1$ and $\alpha>0$,}\\
-\sum_{\jmath=1}^{l-1}\sigma^0\xcirc d^{l-\jmath}\xcirc d^{\jmath}(w_1) &\text{if $1<l$ and $\alpha=0$,}\\
-\sum_{\jmath=0}^{l-1}\sigma^0\xcirc d^{l-\jmath}\xcirc d^{\jmath}(w_1) &\text{if $1<l$ and $\alpha>0$.}
\end{cases}
\end{equation}

\begin{theorem}\label{res nuestra} Let $\Upsilon$ be the family of all the epimorphism of right $E$-modules which split as morphisms of abelian groups. The chain complex
$$
\begin{tikzpicture}
\begin{scope}[yshift=0cm,xshift=0cm]
\matrix(BPcomplex) [matrix of math nodes, row sep=0em, text height=1.5ex, text
depth=0.25ex, column sep=2.5em, inner sep=0pt, minimum height=5mm, minimum width =6mm]
{\mathds{Z} & X_0 & X_1 & X_2 & X_3 & X_4 &\cdots,\\};
\draw[<-] (BPcomplex-1-1) -- node[above=1pt,font=\scriptsize] {$\pi_E$} (BPcomplex-1-2);
\draw[<-] (BPcomplex-1-2) -- node[above=1pt,font=\scriptsize] {$d_1$} (BPcomplex-1-3);
\draw[<-] (BPcomplex-1-3) -- node[above=1pt,font=\scriptsize] {$d_2$} (BPcomplex-1-4);
\draw[<-] (BPcomplex-1-4) -- node[above=1pt,font=\scriptsize] {$d_3$} (BPcomplex-1-5);
\draw[<-] (BPcomplex-1-5) -- node[above=1pt,font=\scriptsize] {$d_4$} (BPcomplex-1-6);
\draw[<-] (BPcomplex-1-6) -- node[above=1pt,font=\scriptsize] {$d_5$} (BPcomplex-1-7);
\end{scope}
\end{tikzpicture}
$$
where $\pi_E$ is the augmentation map, $X_n\coloneqq \bigoplus_{\alpha+\beta=n} X_{\alpha\beta}$ and $d_n$ is the right $E$-module map defined by
$$
d_n(\bx)\coloneqq \begin{cases} \displaystyle{\sum_{l=1}^n d^l_{0n}(\bx)} &\text{if $\bx\in X_{0n}$,}\\ \displaystyle{\sum^{n-\alpha}_{l=0} d^l_{\alpha,n-\alpha}(\bx)} &\text{if $\bx\in X_{\alpha,n-\alpha}$ with $\alpha>0$,}\end{cases}
$$
is a $\Upsilon$-relative projective resolution of $E$.
\end{theorem}

\begin{proof} This is an immediate consequence of~\cite{GG}*{Corollary~A2}.
\end{proof}

\begin{remark} In the previous definition and in the rest of this work we identify each $X_{rs}$ with its image inside $X_{\alpha+\beta}$.
\end{remark}

In order to carry out our computations we also need to give an explicit $\mathds{Z}$-linear contracting homotopy of this resolution. For this we define morphisms of abelian groups
$$
\sigma^l_{l,\beta-l}\colon Y_{\beta}\longrightarrow X_{l,\beta-l}\qquad\text{and}\qquad\sigma^l_{\alpha+l+1,\beta-l}\colon X_{\alpha\beta}\longrightarrow X_{\alpha+l+1,\beta-l},
$$
recursively by $\sigma^l_{\alpha+l+1,\beta-l}\coloneqq  -\sum_{\imath=0}^{l-1}\sigma^0\xcirc d^{l-\imath}\xcirc\sigma^{\imath}$ ($0<l\le \beta$ and $\alpha\ge -1$).

\begin{proposition}\label{cont nuestra} A contracting homotopy $\ov{\sigma}_0\colon E\to X_0$ and $\ov{\sigma}_{n+1}\colon X_n\to X_{n+1}$ ($n\ge 0$), of the resolution introduced in Theorem~\ref{res nuestra}, is given by $\ov{\sigma}_0(\bx)\coloneqq \sigma_{00}^0\xcirc\sigma_0^{-1}(\bx)$ and
$$
\ov{\sigma}_{n+1}(\bx)\coloneqq \begin{cases} \displaystyle{-\sum_{l=0}^{n+1}\sigma_{l,n-l+1}^l\xcirc\sigma_{n+1}^{-1}\xcirc \upsilon_n(\bx) +  \sum_{l=0}^n\sigma_{l+1,n-l}^l(\bx)} &\text{if $\bx\in X_{0n}$,}\\ \displaystyle{\phantom{-} \sum_{l=0}^{n-\alpha}\sigma_{\alpha+l+1,n-\alpha-l}^l(\bx)} &\text{if $\bx\in X_{\alpha,n-\alpha}$ with $\alpha>0$.}\end{cases}
$$
\end{proposition}

\begin{proof} This is a direct consequence of~\cite{GG}*{Corollary~A2}.
\end{proof}

The following theorem gives a closed expression of the homomorphisms $d^l_{\alpha\beta}$ that appear in the relative projective resolution of $E$, obtained above.

\begin{theorem}\label{calculo de los dl} The maps $d^l$ vanish for $l>2$. Moreover
\begin{align*}
&d^1_{\alpha,2\beta-1}(w_1) = (-1)^{\alpha} (w_1 - w_y), && d^2_{2\alpha,\beta}(w_1) = - w_1,\\
&d^1_{\alpha,2\beta}(w_1) = (-1)^{\alpha+1} \sum_{h=0}^{t-1} w_{y^h}, && d^2_{2\alpha+1,\beta}(w_1) = 0.
\end{align*}
\end{theorem}

\begin{proof} We sketch the proof. We first prove the formula for $d^1_{\alpha\beta}$ by induction on $\alpha$. By equality~\eqref{def de dl},
\begin{align*}
& d^1_{0,2\beta+1}(w_1) = -\sigma^0_{0,2\beta}\xcirc\partial_{2\beta+1}\xcirc\upsilon_{2\beta+1}(w_1) = w_1- w_y
\shortintertext{and}
& d^1_{0,2\beta}(w_1) = -\sigma^0_{0,2\beta-1}\xcirc\partial_{2\beta}\xcirc\upsilon_{2\beta}(w_1) = -\sum_{h=0}^{t-1} w_{y^h},
\end{align*}
which proves the case $\alpha=0$. Assume the formula is true for $\alpha$. Then
\begin{align*}
d^1_{\alpha+1,2\beta+1}(w_1) & = -\sigma^0_{\alpha+1,2\beta}\xcirc d^1_{\alpha,2\beta+1}\xcirc d^0_{\alpha+1,2\beta+1}(w_1) = (-1)^{\alpha+1} (w_1 - w_y)
\shortintertext{and}
d^1_{\alpha+1,2\beta}(w_1) & = -\sigma^0_{\alpha+1,2\beta-1}\xcirc d^1_{\alpha,2\beta}\xcirc d^0_{\alpha+1,2\beta}(w_1) = (-1)^{\alpha} \sum_{h=0}^{t-1} w_{y^h},
\end{align*}
as desired. We next prove the formula for $d^2_{\alpha\beta}$. For $\alpha=0$, we have
\begin{equation*}
d^2_{0\beta}(w_1) = -\sigma^0_{1,\beta-2}\xcirc d^1_{0,\beta-1} \xcirc d^1_{0\beta}(w_1) = -w_1.
\end{equation*}
Assume the formula is true for $\alpha$. Then
\begin{equation*}
d^2_{\alpha+1,\beta}(w_1) = -\sigma^0_{\alpha+2,\beta-2}\xcirc \left(d^2_{\alpha\beta}\xcirc d^0_{\alpha+1,\beta} + d^1_{\alpha+1,\beta-1} \xcirc d^1_{\alpha+1,\beta}\right)(w_1) = \begin{cases} \phantom{-} 0 &\text{if $\alpha$ is even,}\\ -w_1 &\text{if $\alpha$ is odd,}\end{cases}
\end{equation*}
as desired. Finally, since
$$
\sigma^0_{\alpha+2,\beta-3}\xcirc d^1_{\alpha+1,\beta-2}\xcirc d^2_{\alpha\beta}(w_1) = \sigma^0_{\alpha+2,\beta-3}\xcirc d^2_{\alpha,\beta-1}\xcirc d^1_{\alpha\beta}(w_1) = \sigma^0_{\alpha+2,\beta-4}\xcirc d^2_{\alpha+1,\beta-2}\xcirc d^2_{\alpha\beta}(w_1) = 0,
$$
from equality~\eqref{def de dl} it follows that $d^l=0$ for $l>2$.
\end{proof}

\begin{proposition}\label{prim prop of ov sigma} The homotopy $\ov{\sigma}$ found in Proposition~\ref{cont nuestra} satisfies
$$
\ov{\sigma}_{n+1}(\bx) = -\sigma_{0,n+1}^0\xcirc\sigma_{n+1}^{-1}\xcirc\upsilon_n(\bx) + \sum_{l=0}^n \sigma_{l+1,n-l}^l(\bx)\qquad\text{for all $\bx\in X_{0n}$.}
$$
\end{proposition}

\begin{proof} By the definitions of $\ov{\sigma}$, $\upsilon$ and $\sigma^{-1}$, it suffices to prove that
\begin{equation}\label{para todo l}
\sigma^l_{l,2\beta-l}(1) =0\quad\text{and}\quad \sigma^l_{l,2\beta+1-l}(y^k) = 0\qquad\text{for all $l\ge 1$ and $0\le k<t-1$.}
\end{equation}
By the definition of $\sigma^l$ and Theorem~\ref{calculo de los dl}, for this it sufficient to consider the cases $l=1$ and $l=2$. We have
\begin{equation*}
\sigma^1_{1,2\beta-1}(1) = -\sigma^0_{1,2\beta-1} \xcirc d^1_{0,2\beta}\xcirc \sigma^0_{0,2\beta}(1)  = 0\quad\text{and}\quad \sigma^1_{1,2\beta}(y^k)  = -\sigma^0_{1,2\beta} \xcirc d^1_{0,2\beta+1}\xcirc \sigma^0_{0,2\beta+1}(y^k)= 0.
\end{equation*}
Therefore,
\begin{equation*}
\sigma^2_{2,2\beta-2}(1) = -\sigma^0_{2,2\beta-2} \xcirc d^2_{0,2\beta}\xcirc \sigma^0_{0,2\beta}(1) = 0 \quad\text{and}\quad \sigma^2_{2,2\beta-1}(y^k) = -\sigma^0_{2,2\beta-1} \xcirc d^2_{0,2\beta-1}\xcirc \sigma^0_{0,2\beta-1}(y^k) = 0,
\end{equation*}
which finishes the proof.
\end{proof}

\begin{remark}\label{rem 2.8} Let  $0\le \jmath<t$. A direct computation shows that
$$
\sigma^0\xcirc\sigma^{-1}\xcirc\upsilon(x^{\imath}w_{y^{\jmath}}) = \begin{cases} \sum_{l=0}^{\jmath-1} w_{y^l} &\text{if $x^{\imath}w_{y^{\jmath}}\in X_{0,2\beta}$,}\\ \delta_{t-1,\jmath} w_1 &\text{if $x^{\imath}w_{y^{\jmath}}\in X_{0,2\beta+1}$.}\end{cases}
$$
\end{remark}

\begin{proposition}\label{prop 2.9} Let $0\le \imath <u$ and $0\le\jmath<t$. For all $\alpha\ge 0$ and $\beta\ge 1$, we have
$$
\sigma^1_{\alpha+2,2\beta}(x^{\imath}w_{y^{\jmath}}) = (-1)^{\alpha+1} \delta_{u-1,\imath}\delta_{t-1,\jmath} w_1\qquad\text{and}\qquad \sigma^1_{\alpha+2,2\beta-1}(x^{\imath}w_{y^{\jmath}})  = (-1)^{\alpha+1} \delta_{u-1,\imath} \sum_{l=0}^{\jmath-1} w_{y^l}.
$$
Moreover, $\sigma^l_{\alpha+l+1,\beta-l} = 0$ for all $l\ge 2$, $\alpha\ge 0$ and $\beta\ge l$.
\end{proposition}

\begin{proof} We sketch the proof. By the definition of $\sigma^1$ above Proposition~\ref{cont nuestra}, we have
\begin{align*}
&\sigma^1_{\alpha+2,2\beta}(x^{\imath}w_{y^{\jmath}})= - \sigma^0_{\alpha+2,2\beta}\xcirc d^1_{\alpha+1,2\beta+1} \xcirc \sigma^0_{\alpha+1,2\beta+1}(x^{\imath}w_{y^{\jmath}}) = (-1)^{\alpha+1} \delta_{u-1,\imath}\delta_{t-1,\jmath} w_1\\
\shortintertext{and}
&\sigma^1_{\alpha+2,2\beta-1}(x^{\imath}w_{y^{\jmath}}) = - \sigma^0_{\alpha+2,2\beta-1}\xcirc d^1_{\alpha+1,2\beta} \xcirc \sigma^0_{\alpha+1,2\beta}(x^{\imath}w_{y^{\jmath}}) = (-1)^{\alpha+1} \delta_{u-1,\imath} \sum_{l=0}^{\jmath-1} w_{y^l},
\end{align*}
which proves the statement for $\sigma^1$. Our next purpose is to prove that $\sigma^2_{\alpha+3,\beta-2} = 0$. We assert that $\sigma^0_{\alpha+3,\beta-2}\xcirc d^2_{\alpha+1,\beta}\xcirc \sigma^0_{\alpha+1,\beta}\! =\! 0$. In fact, if $\alpha$ is even this follows from the fact that $d^2_{\alpha+1,\beta} = 0$, while if $\alpha$ is odd, then the assertion is also true, because
$$
\sigma^0_{\alpha+3,\beta-2}\xcirc d^2_{\alpha+1,\beta}\xcirc \sigma^0_{\alpha+1,\beta}(x^{\imath}w_{y^{\jmath}}) =  - \delta_{u-1,\imath} \sigma^0_{\alpha+3,\beta-2}\bigl(x^{u-1} w_{y^{\jmath}}\bigr)= 0.
$$
Thus,
\begin{equation*}
\sigma^2_{\alpha+3,\beta-2}(x^{\imath}w_{y^{\jmath}}) = - \sigma^0_{\alpha+3,\beta-2}\xcirc d^1_{\alpha+2,\beta-1} \xcirc \sigma^1_{\alpha+2,\beta-1} (x^{\imath}w_{y^{\jmath}}) = 0,
\end{equation*}
as desired. Since, moreover $d^3 = 0$, in order to prove that $\sigma^3_{\alpha+4,\beta-3} = 0$ it suffices to check that the equality $\sigma^0_{\alpha+4,\beta-3} \xcirc d^2_{\alpha+2,\beta-1}\xcirc \sigma^1_{\alpha+2,\beta-1} = 0$ holds. If $\alpha$ is odd this follows from the fact that $d^2_{\alpha+2,\beta-1} = 0$, while if $\alpha$ is even, then a direct computation proves that we also have $\sigma^0_{\alpha+4,\beta-3} \xcirc d^2_{\alpha+2,\beta-1}\xcirc \sigma^1_{\alpha+2,\beta-1}(x^{\imath} w_{y^{\jmath}}) =  0$. The proof that $\sigma^l_{\alpha+l+1,\beta-l} = 0$ for $l\ge 4$, follows easily by induction.
\end{proof}

\subsubsection{Comparison with the normalized bar resolution}\label{comparison maps} Let $(\ov{E}^{\ot *}\ot E,b'_*)$ be the normalized bar resolution of $\mathds{Z}$ as a trivial right $E$-module. It is easy to see that there exist unique morphisms of right $E$-module chain complexes
$$
\phi_*\colon (X_*,d_*)\longrightarrow (\ov{E}^{\ot *}\ot E,b'_*)\qquad\text{and} \qquad\varphi_*\colon (\ov{E}^{\ot *}\ot E,b'_*)\longrightarrow (X_*,d_*),
$$
such that

\begin{itemize}[itemsep=0.7ex,topsep=1.0ex,label=-]

\item $\phi_0 = \varphi_0 = \ide_E$,

\item $\varphi_{n+1}(\bx\ot w_1) = \ov{\sigma}_{n+1}\xcirc \varphi_n\xcirc b'_{n+1}(\bx\ot w_1)$ for all $n\ge 0$ and $\bx\in \ov{E}^{\ot {n+1}}$,

\item the restriction of $\phi_{n+1}$ to $X_{\imath,n+1-\imath}$ satisfies $\phi_{n+1}(w_1) = \ov{\xi}_{n+1}\xcirc \phi_n\xcirc d_{n+1}(w_1)$, where $\ov{\xi}_{n+1}$ is as in subsection~\ref{Hochschild homology of algebras}.

\end{itemize}

\begin{proposition}\label{homotopia} $\varphi_*\xcirc\phi_* =\ide$ and $\phi_*\xcirc \varphi_*$ is homotopically equivalent to the identity map. A homotopy is the one degree map $\omega_{*+1}\colon \phi_*\xcirc\varphi_* \to \ide$, recursively defined by
\begin{equation}
\omega_1\coloneqq  0\quad\text{and}\quad\omega_{n+1}(\byy)\coloneqq \ov{\xi}_{n+1} \xcirc(\phi_n \xcirc\varphi_n -\ide -\omega_n\xcirc b'_n)(\byy)\quad\text{for $n\ge 0$ and $\byy\in \ov{E}^{\ot n}\ot \mathds{Z}w_1$.}\label{aaa}
\end{equation}
Moreover, $\varphi_*\xcirc \omega_* = 0$, $\omega_{*+1}\xcirc \phi_* = 0$ and $\omega_{*+1}\xcirc \omega_* = 0$.
\end{proposition}

\begin{proof} We prove the first two assertions by induction. Clearly $\varphi_0\xcirc\phi_0 =\ide$. Assume that $\varphi_n\xcirc\phi_n =\ide$. Since the image of $\ov{\xi}_{n+1}$ is included in $\ov{E}^{\ot {n+1}}\ot \mathds{Z}w_1$, we have
\begin{align*}
\varphi_{n+1}\xcirc\phi_{n+1}(\byy)&=\ov{\sigma}_{n+1}\xcirc\varphi_n\xcirc b'_{n+1}\xcirc\ov{\xi}_{n+1}\xcirc\phi_n\xcirc d_{n+1}(\byy)\\
& = \ov{\sigma}_{n+1}\xcirc\varphi_n\xcirc\phi_n\xcirc d_{n+1}(\byy) -\ov{\sigma}_{n+1}\xcirc \varphi_n\xcirc\ov{\xi}_n\xcirc b'_n\xcirc\phi_n\xcirc d_{n+1}(\byy)\\
&=\ov{\sigma}_{n+1}\xcirc d_{n+1}(\byy)\\
&= \byy - d_{n+2}\xcirc\ov{\sigma}_{n+2}(\byy),
\end{align*}
for $\byy\in X_{\imath,n+1-\imath}\cap \mathds{Z}w_1$. So, to conclude that $\varphi_{n+1}\xcirc\phi_{n+1} =\ide$ it suffices to check that $\ov{\sigma}_{n+2}(w_1) = 0$, which follows easily from Remark~\ref{rem 2.8} and Propositions~\ref{prim prop of ov sigma} and~\ref{prop 2.9}. Next we prove the second assertion. Clearly $\phi_0\xcirc\varphi_0 - \ide = 0 = b'_1\xcirc \omega_1$. Let $U_n \coloneqq \phi_n\xcirc\varphi_n-\ide$ and $T_n \coloneqq  U_n -\omega_n\xcirc b'_n$. Assuming that $b'_n\xcirc\omega_n +\omega_{n-1}\xcirc b'_{n-1} = U_{n-1}$, we get that
\begin{align*}
b'_{n+1}\xcirc\omega_{n+1}(\byy)+\omega_n\xcirc b'_n(\byy) &= b'_{n+1}\xcirc\ov{\xi}_{n+1}\xcirc T_n(\byy)+\omega_n\xcirc b'_n(\byy)\\
& = T_n(\byy) -\ov{\xi}_n\xcirc b'_n\xcirc T_n(\byy) +\omega_n\xcirc b'_n(\byy)\\
& = U_n(\byy) -\ov{\xi}_n\xcirc b'_n\xcirc U_n(\byy) +\ov{\xi}_n\xcirc b'_n\xcirc\omega_n\xcirc b'_n(\byy)\\
& = U_n(\byy) -\ov{\xi}_n\xcirc U_{n-1}\xcirc b'_n(\byy) +\ov{\xi}_n\xcirc b'_n\xcirc\omega_n\xcirc b'_n(\byy)\\
& = U_n(\byy) -\ov{\xi}_n\xcirc U_{n-1}\xcirc b'_n(\byy) + \ov{\xi}_n\xcirc U_{n-1}\xcirc b'_n(\byy) - \ov{\xi}_n\xcirc \omega_{n-1} \xcirc b'_{n-1}\xcirc b'_n(\byy)\\
&= U_n(\byy),
\end{align*}
for $\byy \in \ov{E}^{\ot n}\ot \mathds{Z}w_1$,
where the first equality holds by identity~\eqref{aaa}; the second one, since $\ov{\xi}$ is a contracting homotopy; the third one, by the definition of $T_n$; the fourth one, since $U_*$ is a morphism; and the fifth one, by the assumption.

It remains to prove the last assertions. We check the last equality assuming that $\omega_{*+1}\xcirc \phi_* = 0$ and $\varphi_*\xcirc\omega_* = 0$, and let the other ones, which are easier, to the reader. It is evident that $\omega_2\xcirc \omega_1 = 0$. Assume that $n\ge 1$ and $\omega_{n+1}\xcirc \omega_n = 0$ and let $\byy\in \ov{E}^{\ot n}\ot \mathds{Z}w_1$. Since
$$
\omega_{n+1}\xcirc b'_{n+1}\xcirc\omega_{n+1}(\byy) = \omega_{n+1}\xcirc(\phi_n\xcirc \varphi_n(\byy)-\byy - \omega_n\xcirc b'_n(\byy)) = - \omega_{n+1}(\byy),
$$
we have
$$
\omega_{n+2}\xcirc \omega_{n+1}(\byy) = \ov{\xi}_{n+2} \xcirc(\phi_{n+1}\xcirc\varphi_{n+1} -\ide - \omega_{n+1}\xcirc b'_{n+1})\xcirc\omega_{n+1}(\byy) = - \ov{\xi}_{n+2} \xcirc (\omega_{n+1}+\omega_{n+1}\xcirc b'_{n+1}\xcirc\omega_{n+1})(\byy) = 0,
$$
as desired.
\end{proof}

For each $\alpha,\beta,n\in \mathds{N}_0$ such that $\alpha+\beta=n$, we let $\varphi_n^{\alpha\beta} \colon E\ot\ov{E}^{\ot n} \to X_{\alpha\beta}$ denote the unique map such that $\varphi_n = \sum_{\alpha+\beta =n} \varphi_n^{\alpha\beta}$.

\begin{remark}\label{calculo de morfismos} A direct computation using Proposition~\ref{prim prop of ov sigma}, Remark~\ref{rem 2.8} and the definitions of $\phi_*$ and $\varphi_*$, shows that
\begin{align*}
& \phi_1(w_1) =  w_y\ot w_1 &&\text{on $X_{01}$},\\
& \phi_1(w_1) = -xw_1\ot w_1 &&\text{on $X_{10}$},\\
& \phi_2 (w_1) = - \sum_{h=1}^{t-1} w_y\ot w_{y^h}\ot w_1 &&\text{on $X_{02}$,}\\
& \phi_2 (w_1) = w_y\ot xw_1\ot w_1 - xw_1\ot w_y\ot w_1 &&\text{on $X_{11}$,}\\
& \phi_2 (w_1) = - \sum_{h=1}^{u-1} xw_1\ot x^hw_1\ot w_1 &&\text{on $X_{20}$},\\
& \varphi_1^{01}(x^{\imath}w_{y^{\jmath}}\ot w_1) = \sum_{h=0}^{\jmath-1} w_{y^h}\\
\shortintertext{and}
& \varphi_1^{10}(x^{\imath}w_{y^{\jmath}}\ot w_1) = - \sum_{h=0}^{\imath-1} x^h w_{y^{\jmath}},
\end{align*}
where $0\le\imath< u$ and $0\le\jmath< t$.
\end{remark}

\begin{remark}\label{calculo de la homotopia} A direct computation shows that
$$
\omega_2(x^{\imath}w_{y^{\jmath}}\ot w_1) = \sum_{h=0}^{\imath-1} xw_1\ot x^h w_{y^{\jmath}} \ot w_1+\sum_{h=1}^{\jmath-1} w_y \ot w_{y^h} \ot w_1,
$$
where $0\le\imath< u$ and $0\le\jmath< t$.
\end{remark}

\subsection{A complex for the homology of cyclic groups}\label{subsection:A complex for the homology of cyclic groups}

Let $v$, $u$, $t$, $C_u \rtimes_{\zeta}C_t$ and $E$ be as in Section~\ref{una resolucion}. Let $D\coloneqq \mathds{Z}[C_v]$ and let $\ov{D}\coloneqq D/\mathds{Z}1$. Recall that the map $f\colon C_u \rtimes_{\zeta}C_t \to C_v$, defined by $f(x^{\imath}w_{y^{\jmath}})\coloneqq g^{t\imath+\jmath}$, where $0\le \jmath<t$, is a group isomorphism. Here we will obtain a chain complex giving the group homology of $C_v$ with coefficients in a commutative group $M$, considered as a left $D$-module via the trivial action (that is $g^{\imath} m = m$). We are interested in the case $M\coloneqq \ov{D}^{\ot s}/\sh(\ov{D}^{\ot s})$ with $s\in \mathds{N}$.

For each $\alpha,\beta\in \mathds{N}_0$, let $M_{\alpha\beta}$ be a copy of $M$. Let $\ov{d}_{\alpha\beta}^l \colon M_{\alpha\beta} \to M_{\alpha+l-1,\beta-l}$ ($\alpha,\beta\ge 0$, $0\le l\le \min(2,\beta)$ and $\alpha+l>0$) be the morphisms defined by:
\begin{equation}\label{pepito}
\begin{aligned}
&\ov{d}^0_{2\alpha-1,\beta}(m) \coloneqq 0, &&\qquad \ov{d}^1_{\alpha,2\beta-1}(m) \coloneqq 0, &&\qquad \ov{d}^2_{2\alpha,\beta}(m) \coloneqq -m,\\
&\ov{d}^0_{2\alpha,\beta}(m) \coloneqq um, &&\qquad \ov{d}^1_{\alpha,2\beta}(m) \coloneqq (-1)^{\alpha+1} tm, &&\qquad\ov{d}^2_{2\alpha+1,\beta}(m) \coloneqq 0.
\end{aligned}
\end{equation}
By the definition of the maps $d^0_{2\alpha-1,\beta}$ and $d^0_{2\alpha,\beta}$, given above Proposition~\ref{las filas son contractiles}, and by Theorem~\ref{calculo de los dl}, tensoring $M$ over $D$ with $(X_*,d_*)$ and using the identifications $\theta_{\alpha\beta}\colon M_{\alpha\beta} \to X_{\alpha\beta}\ot_D M$, given by $\theta_{\alpha\beta}(m) \coloneqq w_1\ot m$, we obtain the chain complex
\begin{equation}\label{a1}
\begin{tikzpicture}[baseline=(current bounding box.center)]
\begin{scope}[yshift=0cm,xshift=0cm]
\matrix(BPcomplex) [matrix of math nodes, row sep=0em, text height=1.5ex, text
depth=0.25ex, column sep=2.5em, inner sep=0pt, minimum height=5mm, minimum width =6mm]
{\ov{X}_0(M) & \ov{X}_1(M) & \ov{X}_2(M) & \ov{X}_3(M) & \ov{X}_4(M) & \ov{X}_5(M) &\cdots,\\};
\draw[<-] (BPcomplex-1-1) -- node[above=1pt,font=\scriptsize] {$\ov{d}_1$} (BPcomplex-1-2);
\draw[<-] (BPcomplex-1-2) -- node[above=1pt,font=\scriptsize] {$\ov{d}_2$} (BPcomplex-1-3);
\draw[<-] (BPcomplex-1-3) -- node[above=1pt,font=\scriptsize] {$\ov{d}_3$} (BPcomplex-1-4);
\draw[<-] (BPcomplex-1-4) -- node[above=1pt,font=\scriptsize] {$\ov{d}_4$} (BPcomplex-1-5);
\draw[<-] (BPcomplex-1-5) -- node[above=1pt,font=\scriptsize] {$\ov{d}_5$} (BPcomplex-1-6);
\draw[<-] (BPcomplex-1-6) -- node[above=1pt,font=\scriptsize] {$\ov{d}_6$} (BPcomplex-1-7);
\end{scope}
\end{tikzpicture}
\end{equation}
where $\ov{X}_n(M)\coloneqq \bigoplus_{\alpha+\beta=n} M_{\alpha\beta}$ and $\ov{d}_n$ is the morphism of abelian groups defined by
\begin{equation}\label{cuentapal}
\ov{d}_n(m)\coloneqq \begin{cases} \displaystyle{\sum_{l=1}^{\min(n,2)} \ov{d}^l_{0n}(m)} &\text{if $m\in M_{0n}$,}\\ \displaystyle{\sum_{l=0}^{\min(n-\alpha,2)} \ov{d}^l_{\alpha,n-\alpha}(m)} &\text{if $m\in M_{\alpha,n-\alpha}$ with $\alpha>0$.}\end{cases}
\end{equation}
Let $\cramped{\bigl(\ov{D}^{\ot *}\ot M,b_*\bigr)}$ be the canonical normalized complex of $C_v$ with coefficients in~$M$.~Re\-call that there is a canonical identification $\bigl(\ov{D}^{\ot *}\ot M,b_*\bigr)\simeq \bigl(\ov{D}^{\ot *}\ot D,b'_*\bigr)\ot_D M$. Let
\begin{equation}\label{mapas basicos}
\ov{\phi}_* \colon (\ov{X}_*(M),\ov{d}_*)\longrightarrow\bigl(\ov{D}^{\ot *}\ot M,b_*\bigr) \quad\text{and}\quad \ov{\varphi}_*\colon \bigl(\ov{D}^{\ot *}\ot M,b_*\bigr)\longrightarrow (\ov{X}_*(M),\ov{d}_*)
\end{equation}
be the morphisms of chain complexes induced by $\phi_*$ and $\varphi_*$, respectively. By definition $\ov{\phi}_0 = \ov{\varphi}_0 = \ide_M$. Moreover, by Proposition~\ref{homotopia} we know that $\ov{\varphi}_*\xcirc \ov{\phi}_* =\ide$ and $\ov{\phi}_*\xcirc\ov{\varphi}_*$ is ho\-motopically equivalent to the identity map. More precisely, a homotopy $\ov{\omega}_{*+1}$, from $\ov{\phi}_*\xcirc\ov{\varphi}_*$ to $\ide$, is the family of maps
\begin{equation}\label{homotop}
\Bigl(\ov{\omega}_{n+1}\colon \ov{D}^{\ot n}\ot M\longrightarrow \ov{D}^{\ot {n+1}}\ot M\Bigr)_{n\ge 0},
\end{equation}
induced by $\cramped{\bigl(\omega_{n+1}\colon \ov{E}^{\ot n}\ot E\longrightarrow \ov{E}^{\ot {n+1}}\ot E\bigr)_{n\ge 0}}$. By Proposition~\ref{homotopia} we also know that $\ov{\omega}_1=0$, $\ov{\varphi}_*\xcirc \ov{\omega}_* = 0$, $\ov{\omega}_{*+1}\xcirc \phi_* = 0$ and $\ov{\omega}_{*+1}\xcirc \ov{\omega}_* = 0$.

\smallskip

For each $\alpha,\beta,n\in \mathds{N}_0$ such that $\alpha+\beta=n$, we let $\ov{\varphi}_n^{\alpha\beta} \colon \ov{D}^{\ot n}\ot M \to M_{\alpha\beta}$ denote the unique map such that $\ov{\varphi}_n = \sum_{\alpha+\beta =n} \ov{\varphi}_n^{\alpha\beta}$. In Section~\ref{Section:Full linear cycle set (co)homology of cyclic cycle sets} we will use the following result with $M\coloneqq \ov{D}^{\ot s}/\sh(\ov{D}^{\ot s})$.

\begin{proposition}\label{basico para mor de comp dob} The following assertions hold:

\begin{enumerate}

\item For each $\alpha,\beta\ge 0$, there exists $x_{\alpha\beta}\in \ov{D}^{\ot {\alpha+\beta}}$ such that $\ov{\phi}_{\alpha+\beta}(m) = x_{\alpha\beta} \ot m$, for all $m\in M_{\alpha\beta}$.

\item For each $\alpha,\beta\ge 0$, there exists a map $\breve{\varphi}_{\alpha+\beta}^{\alpha\beta}\colon \ov{D}^{\ot {\alpha+\beta}}\to \mathds{Z}$ such that
$$
\ov{\varphi}_{\alpha+\beta}^{\alpha\beta}(g^{\imath_1}\ot\cdots\ot g^{\imath_{\alpha+\beta}}\ot m) = \breve{\varphi}_{\alpha+\beta}^{\alpha\beta}  (g^{\imath_1} \ot \cdots\ot g^{\imath_{\alpha+\beta}})m\quad\text{for all $m\in M$.}
$$

\item For each $n\ge 0$, there exists a map $\breve{\omega}_{n+1}\colon \ov{D}^{\ot n}\to \ov{D}^{\ot {n+1}}$ such that
$$
\ov{\omega}_{n+1}(g^{\imath_1}\ot\cdots\ot g^{\imath_n}\ot m) = \breve{\omega}_{n+1}(g^{\imath_1}\ot\cdots\ot g^{\imath_n})\ot m\quad\text{for all $m\in M$.}
$$
\end{enumerate}
\end{proposition}

\begin{proof} All the assertions follow from the fact that the left and right actions of $D$ on $M$ are trivial.
\end{proof}

\begin{remark}\label{calculo de morfismos en hom} By Remark~\ref{calculo de morfismos}, we have
\begin{alignat*}{3}
\begin{aligned}
& \ov{\phi}_1(m) = g\ot m &&\!\!\!\text{on $M_{01}$,}\\[8pt] &\ov{\phi}_1(m) = -g^t\ot m &&\!\!\!\text{on $M_{10}$},
\end{aligned}
\qquad \begin{aligned}
& \ov{\phi}_2 (m) = -\sum_{l=1}^{t-1} g\ot g^l\ot m &&\!\!\!\text{on $M_{02}$,}\\ & \ov{\phi}_2 (m) = g\ot g^t\ot m - g^t\ot g\ot m &&\!\!\!\text{on $M_{11}$,}\\ &\ov{\phi}_2 (m) = - \sum_{l=1}^{u-1} g^t\ot g^{tl}\ot m &&\!\!\!\text{on $M_{20}$,}
\end{aligned}\qquad
\begin{aligned}
& \ov{\varphi}_1^{01}(g^{t\imath+\jmath}\ot m) = \jmath m,\\[8pt] &\ov{\varphi}_1^{10}(g^{t\imath+\jmath}\ot m) = - \imath m,
\end{aligned}
\end{alignat*}
where $0\le\imath< u$ and $0\le\jmath< t$.
\end{remark}

\begin{remark}\label{calculo de la homotopia en hom} By Remark~\ref{calculo de la homotopia}, we have
$$
\ov{\omega}_2(g^{t\imath+\jmath}\ot m) = \sum_{l=0}^{\imath-1} g^t\ot g^{tl+\jmath}\ot m + \sum_{l=1}^{\jmath-1} g \ot g^l\ot m,
$$
where $0\le\imath< u$ and $0\le\jmath< t$.
\end{remark}

\section{Full linear cycle set cohomology of cyclic cycle sets}\label{Section:Full linear cycle set (co)homology of cyclic cycle sets}

Let $p\in \mathds{N}$ be a prime number and let $\nu,\eta\in \mathds{N}$ be such that $0<\nu\le \eta\le 2\nu$. Let $v\coloneqq p^{\eta}$, $u\coloneqq p^{\nu}$, $t\coloneqq p^{\eta-\nu}$ and $u'\coloneqq p^{2\nu-\eta}$. Note that $u't = u$ and $ut=v$. Consider the linear cycle set $\mathcal{A}\coloneqq (\mathds{Z}/v\mathds{Z};\cdot)$, where $\imath\cdot \jmath\coloneqq (1-u\imath)\jmath$. Note that the set of invariants of $\mathcal{A}$ is formed by the multiples of $t$ and that it has $u$ elements. In this section we compute the cohomologies $\Ho_N^1(\mathcal{A},\Gamma)$ and $\Ho_N^2(\mathcal{A},\Gamma)$ of $\mathcal{A}$ with coefficients in an arbitrary abelian group $\Gamma$. Then, using the last result we prove Theorems~A, B and~C of the introduction. Let $C_v\coloneqq \langle g\rangle$ be the multiplicative cyclic group of order $v$, endowed with the binary operation $g^{\imath}\cdot g^{\jmath}\coloneqq g^{\imath\cdot \jmath}$. Let $D\coloneqq \mathds{Z}[C_v]$ and $\ov{D}\coloneqq D/\mathds{Z}1$. Let $\sh(\ov{D}^{\ot s})$ be as in subsection~\ref{The full linear cycle set (co)homology}. For each $r\ge 0$ and $s\ge 1$, let $\ov{M}(s)\coloneqq \ov{D}^{\ot s}/\sh(\ov{D}^{\ot s})$ and let $\ov{X}_{rs} \coloneqq \ov{X}_r(\ov{M}(s))$ where $\ov{X}_r(\ov{M}(s))$ is as in Subsec\-tion~\ref{subsection:A complex for the homology of cyclic groups}. Thus $\ov{X}_{rs} = \bigoplus_{\alpha+\beta = r} \ov{M}(s)_{\alpha \beta}$, where $\alpha,\beta\ge 0$ and each $\ov{M}(s)_{\alpha \beta}$ is a copy of $\ov{M}(s)$. Consider the~dou\-ble complex $(\ov{X}_{**},\ov{d}^{\mathrm{h}}_{**},\ov{d}^{\mathrm{v}}_{**})$, where the $s$-row $(\ov{X}_{*s},\ov{d}^{\mathrm{h}}_{*s})$ is the complex $(\ov{X}_*(\ov{M}(s)),\ov{d}_*)$, introduced in~\eqref{a1}, and $\ov{d}^{\mathrm{v}}_{rs}\coloneqq \bigoplus_{\alpha+\beta = r} \ov{d}^{\mathrm{v}}_{\alpha\beta s}$, in which $\ov{d}^{\mathrm{v}}_{\alpha\beta s}\colon \ov{M}(s)_{\alpha \beta} \to \ov{M}(s-1)_{\alpha \beta}$ is the map defined by
\begin{equation}\label{dvbarra}
\begin{aligned}
\ov{d}^{\mathrm{v}}_{\alpha\beta s}([g^{\imath_1}\ot\cdots\ot g^{\imath_s}]) & \coloneqq (-1)^{r+1} [g^{\imath_2}\ot\cdots\ot g^{\imath_s}]\\
&+\sum_{\jmath=1}^{s-1} (-1)^{\jmath+r+1} [g^{\imath_1}\ot\cdots\ot g^{\imath_{j-1}}\ot g^{\imath_{j}+\imath_{j+1}}\ot g^{\imath_{j+2}}\ot\cdots\ot g^{\imath_{s}}]\\
&+ (-1)^{r+s+1} [g^{\imath_1}\ot\cdots\ot g^{\imath_{s-1}}],
\end{aligned}
\end{equation}
where $[g^{\imath_1}\ot\cdots\ot g^{\imath_s}]$, etcetera, are as in Subsection~\ref{The full linear cycle set (co)homology}. Let $\mathcal{A}_{tr}$ be the group $\mathds{Z}/v\mathds{Z}$, endowed with the trivial structure of linear cycle set. For each $s\ge 1$, let
\begin{equation*}
\ov{\phi}_{*s} \colon (\ov{X}_{*s},\ov{d}^{\mathrm{h}}_{*s})\longrightarrow \bigl(\wh{C}^N_{*s}(\mathcal{A}_{tr}),\partial^{\mathrm{h}}_{*s}\bigr) \qquad\text{and}\qquad \ov{\varphi}_{*s}\colon \bigl(\wh{C}^N_{*s}(\mathcal{A}_{tr}),\partial^{\mathrm{h}}_{*s}\bigr)\longrightarrow (\ov{X}_{*s},\ov{d}^{\mathrm{h}}_{*s})
\end{equation*}
be the maps $\ov{\phi}_*$ and $\ov{\varphi}_*$ introduced in~\eqref{mapas basicos}, with $M$ replaced by $\ov{M}(s)$. By items~(1) and~(2) of Proposition~\ref{basico para mor de comp dob}, in the diagram
\begin{equation}\label{(add1)}
\begin{tikzpicture}[baseline=(current bounding box.center)]
\draw (-0.5,0) node[] {$(\ov{X}_{**},\ov{d}^{\mathrm{h}}_{**},\ov{d}^{\mathrm{v}}_{**})$}; \draw (3.5,0) node {$\bigl(\wh{C}^N_{*s}(\mathcal{A}_{tr}),\partial^{\mathrm{h}}_{**},\partial^{\mathrm{v}}_{**}\bigr)$}; \draw[<-] (0.62,0.125) -- node[above=-2pt,font=\scriptsize] {$\ov{\varphi}_{**}$}(1.9,0.12);\draw[->] (0.62,-0.15) -- node[below=-2pt,font=\scriptsize] {$\ov{\phi}_{**}$}(1.9,-0.12);
\end{tikzpicture}
\end{equation}
the maps $\ov{\phi}_{**}$ and $\ov{\varphi}_{**}$ are morphisms of double complexes. Moreover, we know that $\ov{\varphi}_{**}\xcirc \ov{\phi}_{**} = \ide$, and that in each row $s$, we have a special~de\-for\-ma\-tion retract
\begin{equation}\label{(aaaa1)}
\begin{tikzpicture}[baseline=(current bounding box.center)]
\draw (-0.2,0) node[] {$(\ov{X}_{*s},\ov{d}^{\mathrm{h}}_{*s})$}; \draw (3.1,0) node {$\bigl(\wh{C}^N_{*s}(\mathcal{A}_{tr}),\partial^{\mathrm{h}}_{*s})$};\draw[<-] (0.6,0.12) -- node[above=-2pt,font=\scriptsize] {$\ov{\varphi}_{*s}$}(1.9,0.12);\draw[->] (0.6,-0.12) -- node[below=-2pt,font=\scriptsize] {$\ov{\phi}_{*s}$}(1.9,-0.12); \draw (5.7,0) node {$\wh{C}^N_{*s}(\mathcal{A}_{tr})$}; \draw (9.1,0) node {$\wh{C}^N_{*+1,s}(\mathcal{A}_{tr})$,}; \draw[->] (6.5,0) -- node[above=-2pt,font=\scriptsize] {$\ov{\omega}_{*+1,s}$}(8,0);
\end{tikzpicture}
\end{equation}
where $(\ov{\omega}_{n+1,s})_{n\ge 0}$ is the family of maps $(\ov{\omega}_{n+1})_{n\ge 0}$, introduce in~\eqref{homotop}, with $M$ replaced by $\ov{M}(s)$. For each $s\in \mathds{N}$, we have a perturbation $\delta^h_{*s}\colon \ov{D}^{\ot *}\ot \ov{M}(s)\to \ov{D}^{\ot {*-1}}\ot \ov{M}(s)$, where
\begin{equation}\label{(bb'')}
\begin{aligned}
& \delta^h_{11}(g^{\imath_1}\ot g^{\imath_2}) = g^{\imath_1}\cdot g^{\imath_2} - g^{\imath_2} = g^{(1-u\imath_1)\imath_2} - g^{\imath_2},\\
& \delta^h_{21}(g^{\imath_1}\ot g^{\imath_2}\ot g^{\imath_3})\coloneqq g^{\imath_1}\cdot g^{\imath_2}\ot g^{\imath_1}\cdot g^{\imath_3} - g^{\imath_2} \ot g^{\imath_3} = g^{(1-u\imath_1)\imath_2}\ot g^{(1-u\imath_1)\imath_3} - g^{\imath_2} \ot g^{\imath_3},\\
& \delta^h_{12}(g^{\imath_1}\ot [g^{\imath_2}\ot g^{\imath_3}])\coloneqq [g^{\imath_1}\cdot g^{\imath_2}\ot g^{\imath_1}\cdot g^{\imath_3}] - [g^{\imath_2} \ot g^{\imath_3}] = [g^{(1-u\imath_1)\imath_2}\ot g^{(1-u\imath_1)\imath_3}] - [g^{\imath_2} \ot g^{\imath_3}],\\
& \delta^h_{ns} = 0 \quad\text{for $ns\notin\{11,21,12\}$}.
\end{aligned}
\end{equation}
In order to carry out our computations we are going to apply Proposition~\ref{para complejos dobles} to this data. For this, we first must prove that $\delta^t_*$ is small. Since $\ov{\omega}_{11} = \ov{\omega}_{12} = 0$, the unique non-trivial point is that $\delta^h_{21}\xcirc \ov{\omega}_{21}$ is nilpotent. But, by Remark~\ref{calculo de la homotopia en hom} and the fact that $g^t\cdot g^l=g^l$ and $g\cdot g^l=g^{(1-u)l}=g^{t(u-u'l)+l}$, we have
\begin{equation}\label{formula'}
\delta^h_{21}\xcirc \ov{\omega}_{21} (g^{t\imath+\jmath}\ot g^{\imath_1}) = \sum_{l=1}^{\jmath-1} g\cdot g^l\ot g\cdot g^{\imath_1} - \sum_{l=1}^{\jmath-1} g^l\ot g^{\imath_1} = \sum_{l=1}^{\jmath-1} g^{t(u-u'l)+l}\ot  g^{(1-u)\imath_1} - \sum_{l=1}^{\jmath-1}  g^l \ot g^{\imath_1},
\end{equation}
where $0\le \jmath<t$. Using this it is easy to see that $(\delta^h_{21}\xcirc \ov{\omega}_{21})^{t-1} = 0$.

\begin{remark}\label{util para'} The chain double complex $\bigl(\wh{C}^N_{**}(\mathcal{A}_{tr},\mathds{Z}),\wh{\partial}^{\mathrm{h}}_{**}, \partial^{\mathrm{v}}_{**}\bigr)$, obtained by applying the perturbation~\eqref{(bb'')} to $\bigl(\wh{C}^N_{**}(\mathcal{A}_{tr},\mathds{Z}),\partial^{\mathrm{h}}_{**}, \partial^{\mathrm{v}}_{**}\bigr)$, only coincides with $\bigl(\wh{C}^N_{**}(\mathcal{A},\mathds{Z}),\partial^{\mathrm{h}}_{**},\partial^{\mathrm{v}}_{**}\bigr)$ for $** = 01$, $** = 11$, $** = 21$, $** = 02$, $** = 12$ and $** = 03$. Thus, the chain double complex $\mathcal{X}(\mathcal{A})\coloneqq (\ov{X}_{**},\wh{d}^{\mathrm{h}}_{**}, \wh{d}^{\mathrm{v}}_{**})$, obtained by applying Proposition~\ref{para complejos dobles} to the above data, will be useful only to compute the full (co)homology of $\mathcal{A}$ in degrees $1$ and $2$. Note that $\wh{d}^{\mathrm{v}}_{**}=\ov{d}^{\mathrm{v}}_{**}$.
\end{remark}

For each $\alpha,\beta,r\in \mathds{N}_0$ and $s\in \mathds{N}$ such that $\alpha+\beta=r$, we let $\ov{\varphi}_{rs}^{\alpha\beta} \colon \ov{D}^{\ot r}\ot \ov{M}(s) \to \ov{M}(s)_{\alpha\beta}$ denote the unique map such that $\ov{\varphi}_{rs} = \sum_{\alpha+\beta =r} \ov{\varphi}_{rs}^{\alpha\beta}$. Clearly $\ov{\varphi}_{rs}^{\alpha\beta}$ is the map $\ov{\varphi}_r^{\alpha\beta}$ introduced above Proposition~\ref{basico para mor de comp dob}, with $M$ replaced by $\ov{M}(s)$. A direct computation using equality~\eqref{formula'} and Remark~\ref{calculo de morfismos en hom} shows that
\begin{align*}
& \ov{\varphi}^{01}_{11}\xcirc\delta^h_{21}\xcirc \ov{\omega}_{21}(g^{t\imath+\jmath}\ot g^{\imath_1}) = \sum_{l=1}^{j-1} l (g^{\imath_1}g^{-u\imath_1} - g^{\imath_1}) = \binom{\jmath}{2} g^{\imath_1} (g^{-u\imath_1} - 1)
\shortintertext{and}
& \ov{\varphi}^{10}_{11}\xcirc\delta^h_{21}\xcirc \ov{\omega}_{21}(g^{t\imath+\jmath}\ot g^{\imath_1})
= -\sum_{l=1}^{\jmath-1} (u-u'l) g^{\imath_1} g^{-u\imath_1} = u'\left(\binom{j}{2}-\binom{j-1}{1}t\right) g^{\imath_1} g^{-u\imath_1},
\end{align*}
where $0\le\jmath<t$. An inductive argument using these equalities,~\eqref{formula'} and the fact that $g^{u^2} = 1$ shows that,
\begin{align}
&\ov{\varphi}_{11}^{01}\xcirc(\delta^h_{21}\xcirc \ov{\omega}_{21})^s(g^{t\imath+\jmath}\ot g^{\imath_1}) = \binom{\jmath}{s+1} g^{\imath_1}\bigl(g^{-u\imath_1}-1\bigr)^s\label{prim comp'}
\shortintertext{and}
& \ov{\varphi}_{11}^{10}\xcirc(\delta^h_{21}\xcirc \ov{\omega}_{21})^s(g^{t\imath+\jmath}\ot g^{\imath_1}) = u'\left(\binom{\jmath}{s+1}-\binom{\jmath-1}{s}t\right) g^{\imath_1} g^{-u\imath_1}(g^{-u\imath_1}-1)^{s-1},\label{sec comp'}
\end{align}
for all $0\le\jmath<t$ and $s\in \mathds{N}$. Let $\mathcal{X}(\mathcal{A})_T$ be the subcomplex
$$
\mathcal{X}(\mathcal{A})_T\coloneqq \quad
\begin{tikzcd}[
column sep={4em},
row sep={3.5em},
]
\ov{X}_{03} \arrow[d, "\wh{d}^{\mathrm{v}}_{03}"] \\
\ov{X}_{02} \arrow[d, "\wh{d}^{\mathrm{v}}_{02}"]  & \ov{X}_{12}\arrow[d, "\wh{d}^{\mathrm{v}}_{12}"] \arrow[l,"\wh{d}^{\mathrm{h}}_{12}"',]\\
\ov{X}_{01} & \ov{X}_{11} \arrow[l,"\wh{d}^{\mathrm{h}}_{11}"']& \ov{X}_{21}, \arrow[l,"\wh{d}^{\mathrm{h}}_{21}"']
\end{tikzcd}
$$
of $\mathcal{X}(\mathcal{A})$. Recall that $\ov{X}_{rs} = \bigoplus_{\substack{\alpha,\beta\ge 0  \\\alpha+\beta = r}}\ov{M}(s)_{\alpha\beta}$.

\begin{theorem}\label{complejo para}
The chain double complex $\mathcal{X}(\mathcal{A})_T$ is a partial total complex of the diagram
$$
\mathcal{D}(\mathcal{A})_T\coloneqq \quad
\begin{tikzpicture}[baseline=(a).base]
\node[scale=0.8](a)at(0,0){
\begin{tikzcd}[row sep=5em, column sep=1.5em] 	
&&\ov{M}(1)_{02} \arrow[ddl,"\wh{d}^{\mathrm{h}1}_{021}"' pos=0.35, outer sep=-1pt] \arrow[ddddr, end anchor={[xshift=-1ex]},,"\wh{d}^{\mathrm{h}2}_{021}" pos=0.39, outer sep=-1pt] \\
&\ov{M}(2)_{01} \arrow[ddl,"\wh{d}^{\mathrm{h}1}_{012}"' pos=0.35, outer sep=-1pt] \arrow[d,"\wh{d}^{\mathrm{v}}_{012}" pos=0.20, outer sep=-1pt] \\
\ov{M}(3)_{00} \arrow[d,"\wh{d}^{\mathrm{v}}_{003}" pos=0.20, outer sep=-1pt] &\ov{M}(1)_{01} \arrow[ddl,"\wh{d}^{\mathrm{h}1}_{011}"' pos=0.35, outer sep=-1pt]  &&& \ov{M}(1)_{11} \arrow[ddl,end anchor={[xshift=0.5ex]},"\wh{d}^{\mathrm{h}1}_{111}"' pos=0.35, outer sep=-1pt] \arrow[lll, crossing over,"0"' pos=0.65, outer sep=-1pt] \\
\ov{M}(2)_{00}\arrow[d,"\wh{d}^{\mathrm{v}}_{002}" pos=0.20, outer sep=-1pt] &&& \ov{M}(2)_{10} \arrow[lll, crossing over,"0"', outer sep=-1pt] \arrow[d,"\wh{d}^{\mathrm{v}}_{102}" pos=0.20, outer sep=-1pt]\\
\ov{M}(1)_{00} &&& \ov{M}(1)_{10} \arrow[lll,"0"', outer sep=-1pt] &&& \ov{M}(1)_{20}, \arrow[lll, "\wh{d}^{\mathrm{h}0}_{201}"', outer sep=-1pt]
\end{tikzcd}
};
\end{tikzpicture}
$$
where $\wh{d}^{\mathrm{v}}_{***} = \ov{d}^{\mathrm{v}}_{***}$ (see formula~\eqref{dvbarra}), and the other not zero maps are given by:
\begin{alignat*}{2}
& \wh{d}^{\mathrm{h}0}_{201}(g^{\imath})\coloneqq ug^{\imath}, &&\qquad \wh{d}^{\mathrm{h}1}_{011}(g^{\imath})\coloneqq g^{\imath}(g^{-u\imath} - 1),\\
&\wh{d}^{\mathrm{h}1}_{111}(g^{\imath})\coloneqq g^{\imath}(1-g^{-u\imath}), &&\qquad  \wh{d}^{\mathrm{h}1}_{021}(g^{\imath})\coloneqq - g^{\imath} \sum_{s=0}^{t-1} g^{-su\imath},\\
& \wh{d}^{\mathrm{h}1}_{012}([g^{\imath_1}\ot g^{\imath_2}]) \coloneqq [g^{(1-u)\imath_1}\ot g^{(1-u)\imath_2}]-[g^{\imath_1}\ot g^{\imath_2}], &&\qquad \wh{d}^{\mathrm{h}2}_{021}(g^{\imath})\coloneqq -g^{\imath} + u' g^{\imath} \left(\sum_{s=1}^{t-1} s g^{-u\imath s} \right).
\end{alignat*}

\end{theorem}

\begin{proof} In order to prove this theorem we will apply Proposition~\ref{para complejos dobles} to the data consisting of diagrams~\eqref{(add1)} and~\eqref{(aaaa1)}. We begin by computing the first row of $\mathcal{X}(\mathcal{A})_T$. Since $\ov{\omega}_{11} = 0$ and $\ov{\varphi}_{01} = \ide_{\ov{M}(1)}$ we know that $\wh{d}^{\mathrm{h}}_{11} = \ov{d}^{\mathrm{h}}_{11}+\delta^h_{11}\xcirc \ov{\phi}_{11}$. Moreover, by Remark~\ref{calculo de morfismos en hom} and the definition of $\delta^h_{11}$, we have
$$
\delta^h_{11}\xcirc \ov{\phi}_{11}(g^{\imath}) = 0\text{ on $\ov{M}(1)_{10}$}\quad\text{and}\quad \delta^h_{11}\xcirc \ov{\phi}_{11}(g^{\imath}) = \delta^h_{11}(g\ot g^{\imath}) = g^{(1-u)\imath}- g^{\imath}\text{ on $\ov{M}(1)_{01}$.}
$$
Since $\ov{d}^{\mathrm{h}}_{11} = 0$ (by equalities~\eqref{pepito} and~\eqref{cuentapal}), this implies that
$$
\wh{d}^{\mathrm{h}}_{11}(g^{\imath}) = 0 \text{ on $\ov{M}(1)_{10}$}\qquad\text{and}\qquad \wh{d}^{\mathrm{h}}_{11}(g^{\imath}) = g^{\imath}(g^{-u\imath}- 1) = \wh{d}^{\mathrm{h}1}_{011}(g^{\imath}) \text{ on $\ov{M}(1)_{01}$.}
$$
We next compute $\wh{d}^{\mathrm{h}}_{21}$. By Remark~\ref{calculo de morfismos en hom} and the definition of $\delta^h_{21}$, we have
\begin{align}
& \delta^h_{21}\xcirc \ov{\phi}_{21}(g^{\imath}) = 0 &&\text{ on $\ov{M}(1)_{20}$}\label{eqq1},\\
& \delta^h_{21}\xcirc \ov{\phi}_{21}(g^{\imath}) = g^{(1-u)t}\ot g^{(1-u)\imath} - g^t\ot g^{\imath} = g^t\ot g^{\imath}\bigl(g^{-u\imath} - 1\bigr)  &&\text{ on $\ov{M}(1)_{11}$}\label{eqq2},\\
& \delta^h_{21}\xcirc \ov{\phi}_{21}(g^{\imath}) = \sum_{l=1}^{t-1}  g^l\ot g^{\imath} - \sum_{l=1}^{t-1} g^{t(u-u'l)+l}\ot g^{\imath}g^{-u\imath} &&\text{ on $\ov{M}(1)_{02}$.}\label{eqq3}
\end{align}
Hence, again by Remark~\ref{calculo de morfismos en hom},
\begin{align}
& \ov{\varphi}^{01}_{11}\xcirc \delta^h_{21}\xcirc \ov{\phi}_{21}(g^{\imath}) = 0 && \text{on $\ov{M}(1)_{11}$,}\label{1de4}\\
&\ov{\varphi}^{10}_{11}\xcirc \delta^h_{21}\xcirc \ov{\phi}_{21}(g^{\imath}) = g^{\imath}(1-g^{-u\imath}) && \text{on $\ov{M}(1)_{11}$,}\label{2de4}\\
& \ov{\varphi}_{11}^{01}\xcirc\delta^h_{21}\xcirc \ov{\phi}_{21}(g^{\imath}) = - \binom{t}{2}g^{\imath} \bigl(g^{-u\imath}-1\bigr) && \text{on $\ov{M}(1)_{02}$,}\label{3de4}\\
&\ov{\varphi}_{11}^{10}\xcirc\delta^h_{21}\xcirc \ov{\phi}_{21}(g^{\imath}) = u'\binom{t}{2}g^{\imath} g^{-u\imath} && \text{on $\ov{M}(1)_{02}$}\label{4de4}.
\end{align}
By equality~\eqref{eqq1} we know that $\wh{d}^{\mathrm{h}}_{21}(g^{\imath}) = \ov{d}^{\mathrm{h}}_{21}(g^{\imath})$ on $\ov{M}(1)_{20}$. Consequently, by equalities~\eqref{pepito} and~\eqref{cuentapal}, we have $\wh{d}^{\mathrm{h}}_{21}(g^{\imath}) = \wh{d}^{\mathrm{h}0}_{201}(g^{\imath})$ on $\ov{M}(1)_{20}$. Moreover,~by~equal\-ities~\eqref{prim comp'},~\eqref{sec comp'} and~\eqref{eqq2},
$$
\ov{\varphi}_{11}^{01}\xcirc\bigl(\delta^h_{21}\xcirc \ov{\omega}_{21}\bigr)^s\xcirc \delta^h_{21}\xcirc \ov{\phi}_{21}(g^{\imath}) = 0\text{ on $\ov{M}(1)_{11}$}\quad\text{and}\quad \ov{\varphi}_{11}^{10}\xcirc\bigl(\delta^h_{21}\xcirc \ov{\omega}_{21}\bigr)^s\xcirc\delta^h_{21}\xcirc \ov{\phi}_{21}(g^{\imath})= 0\text{ on $\ov{M}(1)_{11}$}\qquad\text{for all $s\ge 1$,}
$$
while, by equalities~\eqref{pepito} and~\eqref{cuentapal}, we know that $\ov{d}^{\mathrm{h}}_{21}=0$ on $\ov{M}(1)_{11}$. Hence, by equalities~\eqref{1de4} and~\eqref{2de4},
$$
\wh{d}^{\mathrm{h}}_{21}(g^{\imath}) = \ov{\varphi}^{01}_{11}\xcirc \delta^h_{21}\xcirc \ov{\phi}_{21}(g^{\imath}) + \ov{\varphi}^{10}_{11}\xcirc \delta^h_{21}\xcirc \ov{\phi}_{21}(g^{\imath})= \wh{d}^{\mathrm{h}1}_{111}(g^{\imath})\qquad\text{on $\ov{M}(1)_{11}$.}
$$
We now compute $\wh{d}^{\mathrm{h}}_{21}$ on $\ov{M}(1)_{02}$. Equalities~\eqref{prim comp'} and~\eqref{sec comp'} implies that
\begin{align*}
& \ov{\varphi}_{11}^{01}\xcirc\bigl(\delta^h_{21}\xcirc \ov{\omega}_{21}\bigr)^s\biggl(\sum_{l=1}^{t-1} g^l\ot g^{\imath}\biggr) = \sum_{l=1}^{t-1} \binom{l}{s+1} g^{\imath}\bigl(g^{-u\imath}-1\bigr)^s = \binom{t}{s+2} g^{\imath}\bigl(g^{-u\imath}-1\bigr)^s,\\
& \ov{\varphi}_{11}^{01}\xcirc\bigl(\delta^h_{21}\xcirc \ov{\omega}_{21}\bigr)^s\biggl(\sum_{l=1}^{t-1} g^{t(u-u'l)+l}\ot g^{\imath} g^{-u\imath}\biggr)
= \binom{t}{s+2} g^{\imath}g^{-u\imath}\bigl(g^{-u\imath}-1\bigr)^s,\\
&\ov{\varphi}_{11}^{10}\xcirc\bigl(\delta^h_{21}\xcirc \ov{\omega}_{21}\bigr)^s\biggl(\sum_{l=1}^{t-1} g^l\ot g^{\imath}\biggr)= \sum_{l=1}^{t-1} u'\left(\binom{l}{s+1}-\binom{l-1}{s}t\right) g^{\imath} g^{-u\imath}(g^{-u\imath}-1)^{s-1}\\
&\phantom{\ov{\varphi}_{11}^{10}\xcirc\bigl(\delta^h_{21}\xcirc \ov{\omega}_{21}\bigr)^s\biggl(\sum_{l=1}^{t-1} g^l\ot g^{\imath}\biggr)} = u'\left(\binom{t}{s+2}-\binom{t-1}{s+1}t\right) g^{\imath} g^{-u\imath}(g^{-u\imath}-1)^{s-1}\\
&\phantom{\ov{\varphi}_{11}^{10}\xcirc\bigl(\delta^h_{21}\xcirc \ov{\omega}_{21}\bigr)^s\biggl(\sum_{l=1}^{t-1} g^l\ot g^{\imath}\biggr)} = -u'(s+1)\binom{t}{s+2} g^{\imath} g^{-u\imath}(g^{-u\imath}-1)^{s-1},\\
& \ov{\varphi}_{11}^{10}\xcirc\bigl(\delta^h_{21}\xcirc \ov{\omega}_{21}\bigr)^s\biggl(\sum_{l=1}^{t-1} g^{t(u-u'l)+l}\ot  g^{\imath}g^{-u\imath}\biggr) = -u'(s+1)\binom{t}{s+2} g^{\imath}g^{-2u\imath}(g^{-u\imath}-1)^{s-1},
%
%
%
\end{align*}
for all $s\ge 1$ (in the computation of last equality we had used that $g^{u^2}=1$). So, by equality~\eqref{eqq3},
\begin{align}
&\ov{\varphi}_{11}^{01}\xcirc\bigl(\delta^h_{21}\xcirc \ov{\omega}_{21}\bigr)^s\xcirc \delta^h_{21}\xcirc \ov{\phi}_{21}(g^{\imath}) = - \binom{t}{s+2} g^{\imath}\bigl(g^{-u\imath}-1\bigr)^{s+1} &&\text{on $\ov{M}(1)_{02}$,}\label{eqq4}\\
&\ov{\varphi}_{11}^{10}\xcirc\bigl(\delta^h_{21}\xcirc \ov{\omega}_{21}\bigr)^s\xcirc \delta^h_{21}\xcirc \ov{\phi}_{21}(g^{\imath}) = u'(s+1) \binom{t}{s+2} g^{\imath}g^{-u\imath}\bigl(g^{-u\imath}-1\bigr)^s &&\text{on $\ov{M}(1)_{02}$,}\label{eqq5}
\end{align}
for all $s\ge 1$. Thus, by~\eqref{3de4} and~\eqref{eqq4},
\begin{equation}\label{eqqqq1}
\sum_{s=0}^{t-2} \ov{\varphi}_{11}^{01}\xcirc\bigl(\delta^h_{21}\xcirc \ov{\omega}_{21}\bigr)^s\xcirc \delta^h_{21}\xcirc \ov{\phi}_{21}(g^{\imath}) = - \sum_{s=1}^{t-1} \binom{t}{s+1} g^{\imath}\bigl(g^{-u\imath}-1\bigr)^s = - g^{\imath} \left(1-t + \sum_{s=1}^{t-1} g^{-su\imath}\right)
\end{equation}
on $\ov{M}(1)_{02}$. (The last equality follows by induction on $t$ using that $\binom{t}{s+1} = \binom{t-1}{s}+\binom{t-1}{s+1}$).
A similar argument using~\eqref{4de4} and~\eqref{eqq5}, shows that, on $\ov{M}(1)_{02}$ we have
$$
\sum_{s=0}^{t-2} \ov{\varphi}_{11}^{10}\xcirc\bigl(\delta^h_{21}\xcirc \ov{\omega}_{21}\bigr)^s\xcirc \delta^h_{21}\xcirc \ov{\phi}_{21}(g^{\imath}) = \sum_{s=0}^{t-2} u'(s+1) \binom{t}{s+2} g^{\imath}g^{-u\imath}\bigl(g^{-u\imath}-1\bigr)^s = u' g^{\imath} \left(\sum_{s=0}^{t-2} (s+1) g^{-u\imath (s+1)} \right).
$$
Combining this with~\eqref{pepito} and~\eqref{cuentapal}, we obtain that, on $\ov{M}(1)_{02}$,
$$
\wh{d}^{\mathrm{h}}_{21}(g^{\imath}) = \ov{d}^{\mathrm{h}}_{21}(g^{\imath}) - g^{\imath} \left(1-t + \sum_{s=1}^{t-1} g^{-su\imath}\right) +  u' g^{\imath} \left(\sum_{s=0}^{t-2} (s+1) g^{-u\imath (s+1)} \right) = \wh{d}^{\mathrm{h}1}_{021}(g^{\imath}) + \wh{d}^{\mathrm{h}2}_{021}(g^{\imath}).
$$
We now com\-pute the second row of $\mathcal{X}(\mathcal{A})_T$. Since $\ov{\omega}_{12} = 0$ and $\ov{\varphi}_{02} = \ide_{\ov{M}(2)}$, we know that $\wh{d}^{\mathrm{h}}_{12} = \ov{d}^{\mathrm{h}}_{12}+\delta^h_{12}\xcirc \ov{\phi}_{12}$. Moreover, by Remark~\ref{calculo de morfismos en hom} and the definition of $\delta^h_{12}$,
$$
\delta^h_{12}\xcirc \ov{\phi}_{12}([g^{\imath_1}\ot g^{\imath_2}]) = 0\text{ on $\ov{M}(2)_{10}$}\quad\text{and}\quad \delta^h_{12}\xcirc \ov{\phi}_{12}([g^{\imath_1}\ot g^{\imath_2}]) = [g^{(1-u)\imath_1}\ot g^{(1-u)\imath_2}]-[g^{\imath_1}\ot g^{\imath_2}]\text{ on $\ov{M}(2)_{01}$.}
$$
Therefore, by equalities~\eqref{pepito} and~\eqref{cuentapal},
$$
\wh{d}^{\mathrm{h}}_{12}([g^{\imath_1}\ot g^{\imath_2}]) = 0\text{ on $\ov{M}(2)_{10}$}\quad\text{and}\quad \wh{d}^{\mathrm{h}}_{12}([g^{\imath_1}\ot g^{\imath_2}]) = \wh{d}^{\mathrm{h}1}_{012}([g^{\imath_1}\ot g^{\imath_2}]) \text{ on $\ov{M}(2)_{01}$,}
$$
as desired.
\end{proof}

\begin{remark}\label{mascalculos}
Let $\bigl(\wh{C}^N_*(\mathcal{A},\mathds{Z}),\partial_*\bigr)$ be as in subsection~\ref{The full linear cycle set (co)homology}, let $\Tot(\mathcal{X}(\mathcal{A}))$ be the total complex of $\mathcal{X}(\mathcal{A})$ and let $\wh{\varphi}_{**}\colon \bigl(\wh{C}^N_{**}(\mathcal{A},\mathds{Z}),\partial^{\mathrm{h}}_{**}, \partial^{\mathrm{v}}_{**}\bigr) \longrightarrow \mathcal{X}(\mathcal{A})$ be the map obtained by applying Proposition~\ref{para complejos dobles} to the special de\-for\-ma\-tion retracts~\eqref{(aaaa1)}, endowed with the perturbations $\delta^h_{*s}$ given in~\eqref{(bb'')}. By that proposition, the map $\wh{\varphi}_*\colon \bigl(\wh{C}^N_*(\mathcal{A},\mathds{Z}), \partial_*\bigr)\to \Tot(\mathcal{X}(\mathcal{A}))$, induced by $\wh{\varphi}_{**}$, is an homotopy equivalence. Since $\ov{\omega}_{11}=0$ and $\ov{\omega}_{12}=0$, we have $\wh{\varphi}_{01} = \ov{\varphi}_{01} = \ide_{\ov{M}(1)}$ and $\wh{\varphi}_{02} = \ov{\varphi}_{02} = \ide_{\ov{M}(2)}$. On the other hand, by Remark~\ref{calculo de morfismos en hom} and equalities~\eqref{prim comp'} and~\eqref{sec comp'}, for each $0\le \imath < u$ and $0\le \jmath<t$, we have
$$
\wh{\varphi}_{11}(g^{t\imath+\jmath}\ot g^{\imath_1}) = \wh{\varphi}_{11}^{01}(g^{t\imath+\jmath}\ot g^{\imath_1}) + \wh{\varphi}_{11}^{10}(g^{t\imath+\jmath}\ot g^{\imath_1}),
$$
where $\wh{\varphi}_{11}^{01}\colon \ov{D}\ot \ov{M}(1)\longrightarrow \ov{M}(1)_{01}$ and $\wh{\varphi}_{11}^{10}\colon \ov{D}\ot \ov{M}(1)\longrightarrow \ov{M}(1)_{10}$ are the maps given by
\begin{align}
& \wh{\varphi}_{11}^{01}(g^{t\imath+\jmath}\ot g^{\imath_1}) \coloneqq \sum_{s=0}^{\jmath-1} \binom{\jmath}{s+1} g^{\imath_1}\bigl(g^{-u\imath_1}-1\bigr)^s \label{eqqq1'}
\shortintertext{and}
&\wh{\varphi}_{11}^{10}(g^{t\imath+\jmath}\ot g^{\imath_1}) \coloneqq -\imath g^{\imath_1} + \sum_{s=1}^{\jmath-1} u'\left(\binom{\jmath}{s+1}-\binom{\jmath-1}{s}t\right) g^{\imath_1} g^{-u\imath_1}(g^{-u\imath_1}-1)^{s-1}.\label{eqqq2'}
\end{align}
\end{remark}

\begin{proposition}\label{funciones de arriba} For all $0\le \imath< u$ and $0\le \jmath< t$, the following identities hold:
$$
\wh{\varphi}_{11}^{01}(g^{t\imath+\jmath}\ot g^{\imath_1}) = g^{\imath_1}\sum_{l=0}^{\jmath-1} g^{-ul\imath_1}\qquad\text{and}\qquad \wh{\varphi}_{11}^{10}(g^{t\imath+\jmath}\ot g^{\imath_1}) = -\imath g^{\imath_1} + u'g^{\imath_1}\sum_{l=1}^{\jmath-1} (\jmath-l-t) g^{-ul\imath_1}.
$$
\end{proposition}

\begin{proof} The first equality holds by the last equality in~\eqref{eqqqq1}. By this and~\eqref{eqqq2'} in order to prove the second one it suffices to show that
$$
\sum_{s=1}^{\jmath-1} \binom{\jmath}{s+1} g^{-u\imath_1}(g^{-u\imath_1}-1)^{s-1} = \sum_{l=1}^{\jmath-1} (\jmath-l) g^{-ul\imath_1}.
$$
But this follows by induction on $\jmath$ using that $\binom{\jmath}{s+1} = \binom{\jmath-1}{s} + \binom{\jmath-1}{s+1}$.
\end{proof}

\subsection{Computing the full linear cycle set cohomology}\label{Computing the full linear cycle set cohomology}
Here we will use freely the notations introduced in subsection~\ref{The full linear cycle set (co)homology}. Let $\mathcal{A}$ be as at the beginning of this section and let $\Gamma$ be an additive abelian group. In this subsection we compute $\Ho_N^2(\mathcal{A},\Gamma)$. Moreover, we obtain a family of $2$-cocycles of $\bigl(C_N^*(\mathcal{A},\Gamma),\partial^*\bigr)$ that applies surjectively on $\Ho_N^2(\mathcal{A},\Gamma)$, and we determine when two of these cocycles are cohomologous. By \cite{LV}*{Theorem~5.8} this classify the central extensions of $\mathcal{A}$ by $\Gamma$. We use this fact in order to prove Theorems~A, B~and C.

\medskip

For each $l\in \mathds{N}$ and $\alpha,\beta\in \mathds{N}_0$ we define $\Gamma(l)^{\alpha\beta}\coloneqq \Hom(\ov{M}(l)_{\alpha\beta},\Gamma)$. There are obvious identifications
\begin{align*}
& \Gamma(1)^{\alpha\beta} = \left\{\sum_{0\le {\imath}<v} \gamma_{\imath}  g^{\imath} : \gamma_{\imath}\in \Gamma\text{ for all ${\imath}$ and } \gamma_0 = 0\right\}
\shortintertext{and}
&\Gamma(2)^{\alpha\beta} = \left\{\sum_{0\le \imath,\jmath<v} \gamma_{\imath \jmath}  g^{\imath}\ot g^{\jmath} : \gamma_{\imath\jmath}\in \Gamma\text{ for all $\imath,\jmath$, } \gamma_{\imath\jmath}=\gamma_{\jmath\imath}\text{ and } \gamma_{0\jmath} = 0\text{ for all $\jmath$}\right\}.
\end{align*}
Let $\mathcal{X}(\mathcal{A},\Gamma)\coloneqq (\ov{X}^{**},\wh{d}_{\mathrm{h}}^{**},\wh{d}_{\mathrm{v}}^{**})$ be the cochain double complex obtained by applying the functor $\Hom(-,\Gamma)$ to $\mathcal{X}(\mathcal{A})_T$. In Proposition~\ref{H1} and Theorems~\ref{caso t=1},~\ref{caso 1<t<u} and~\ref{caso t=u=p}, we are going to calculate $\Ho_N^1(\mathcal{A},\Gamma)$ and $\Ho_N^2(\mathcal{A},\Gamma)$. By Remarks~\ref{son iguales2} and~\ref{util para'} in order to carry out this task we must compute $\Ho^1(\mathcal{X}(\mathcal{A},\Gamma))$ and $\Ho^2(\mathcal{X}(\mathcal{A},\Gamma))$. For this we will use strongly that $\mathcal{X}(\mathcal{A},\Gamma)$ is a partial total complex of the diagram $\mathcal{D}(\mathcal{A},\Gamma)^T\coloneqq \Hom(\mathcal{D}(\mathcal{A})_T,\Gamma)$. In particular
$$
\ov{X}^{01}=\Gamma(1)^{00},\quad \ov{X}^{02}=\Gamma(2)^{00}\quad\text{and}\quad \ov{X}^{11}=\Gamma(1)^{01}\oplus \Gamma(1)^{10}.
$$
We also want to obtain $2$-cocycles of $\wh{C}_N^*(\mathcal{A},\Gamma)$ that represent all the elements of $\Ho_N^2(\mathcal{A},\Gamma)$. These cocycles are obtained by applying $\Hom(\wh{\varphi}_2,\Gamma)$ to $2$-cocycles of $\mathcal{X}(\mathcal{A},\Gamma)$, where
$$
\wh{\varphi}_2\colon \ov{M}(2)\oplus (\ov{D}\ot\ov{M}(1))\longrightarrow \ov{X}_{02}\oplus \ov{X}_{11} = \ov{M}(2)_{00}\oplus \ov{M}(1)_{01} \oplus \ov{M}(1)_{10},
$$
is as in Remark~\ref{mascalculos}. Thus $\wh{\varphi}_2$ is given by the matrix
\begin{equation}\label{matriz}
[\wh{\varphi}_2]\coloneqq \begin{pmatrix} \ide_{\ov{M}(2)} & 0 \\ 0 & \wh{\varphi}_{11}^{01}\\ 0 & \wh{\varphi}_{11}^{10}  \end{pmatrix}.
\end{equation}
From now on we set $\wh{d}^{\alpha\beta s}_{\mathrm{v}}\coloneqq \Hom\bigl(\wh{d}_{\alpha\beta s}^{\mathrm{v}},\Gamma\bigr)$ and $\wh{d}^{\alpha\beta s}_{\mathrm{h}i}\coloneqq \Hom\bigl(\wh{d}_{\alpha\beta s}^{\mathrm{h}i},\Gamma\bigr)$.

\begin{remark}\label{pero} Applying  the functor $\Gamma\ot -$ to $\mathcal{X}(\mathcal{A})_T$ we obtain a chain double complex that gives $\Ho_1^N(\mathcal{A},\Gamma)$ and $\Ho_2^N(\mathcal{A},\Gamma)$. But we are not interested in the computation of these homology groups in this paper.
\end{remark}

Here and subsequently, we adopt the convention that $\gamma_{\imath\jmath}\coloneqq \gamma_{\imath'\jmath'}$, where $\imath'$ and $\jmath'$ are the remainder of the integer division of $\imath$ by $v$ and $\jmath$ by $v$, respectively.

\begin{proposition}\label{H1} We have $\Ho_N^1(\mathcal{A},\Gamma) = \wh{\Ho}_N^1(\mathcal{A},\Gamma)\simeq \Gamma_u$.
\end{proposition}

\begin{proof} By Remarks~\ref{son iguales2} and~\ref{util para'} we know that $\Ho_N^1(\mathcal{A},\Gamma) = \wh{\Ho}_N^1(\mathcal{A},\Gamma) = \Ho^1(\mathcal{X}(\mathcal{A},\Gamma))$. By definition
$$
\wh{d}^{002}_{\mathrm{v}}\left(\sum \gamma_{\imath}g^{\imath}\right) = \sum_{\imath,\jmath} (\gamma_{\imath+\jmath}-\gamma_{\imath}-\gamma_{\jmath}) g^{\imath}\ot g^{\jmath}\quad\text{and}\quad \wh{d}^{011}_{\mathrm{h}1}\left(\sum \gamma_{\imath}g^{\imath}\right) = \sum_{\imath} \left(\gamma_{\imath-u\imath} - \gamma_{\imath}\right)g^{\imath}.
$$
From the first equality we get
$$
\ker(\wh{d}^{002}_{\mathrm{v}}) = \left\{\sum \gamma_{\imath}g^{\imath}: \gamma_{\imath} = \imath \gamma_1 \text{ and } v\gamma_1 = \gamma_v = 0 \right\}.
$$
Consequently, $-u\gamma_1 = (1-u)\gamma_1 -\gamma_1 = \gamma_{1-u} - \gamma_1$, and so
$$
\Ho^1(\mathcal{X}(\mathcal{A},\Gamma)) = \ker(\wh{d}^{002}_{\mathrm{v}})\cap  \ker(\wh{d}^{011}_{\mathrm{h}1}) = \left\{\sum \gamma_{\imath}g^{\imath}: \gamma_{\imath} = \imath \gamma_1 \text{ and } u\gamma_1 = 0 \right\},
$$
which is clearly isomorphic to $\Gamma_{\!u}$.
\end{proof}

Our next purpose is to compute $\wh{\Ho}_N^2(\mathcal{A},\Gamma)$.

\begin{lemma}\label{lema 5.3} $\sum \gamma_{\imath\jmath}  g^{\imath}\ot g^{\jmath} \in \ker(\wh{d}^{003}_{\mathrm{v}})$ if and only if $\gamma_{\imath\jmath} = \sum_{k=\jmath}^{\imath+\jmath-1} \gamma_{1k} - \sum_{k=1}^{\imath-1} \gamma_{1k}$ for $1\le \imath,\jmath<v$.
\end{lemma}

\begin{proof} Assume that $\sum \gamma_{\imath\jmath}  g^{\imath}\ot g^{\jmath} \in \ker(\wh{d}^{003}_{\mathrm{v}})$. Then, for all $a\le b$, we have
$$
0 = \wh{d}^{003}_{\mathrm{v}}\Bigl(\sum \gamma_{\imath\jmath}  g^{\imath}\ot g^{\jmath}\Bigl)(g\ot g^a\ot g^b) = -\gamma_{ab} + \gamma_{a+1,b} - \gamma_{1,a+b} + \gamma_{1a}.
$$
Thus, $\gamma_{a+1,b} = \gamma_{ab} + \gamma_{1,a+b} - \gamma_{1a}$. An inductive argument using this fact proves that the statement is true when $a\le b$. For $a>b$, we have
$$
\gamma_{ab} = \gamma_{ba} = \sum_{a\le k <a+b} \gamma_{1k} - \sum_{1\le k <b} \gamma_{1k} = \sum_{b\le k <a+b} \gamma_{1k} - \sum_{1\le k <a} \gamma_{1k}.
$$
Conversely, assume that $\gamma_{\imath\jmath} = \sum_{k=\jmath}^{\imath+\jmath-1} \gamma_{1k} - \sum_{k=1}^{\imath-1} \gamma_{1k}$ for $1\le \imath,\jmath<v$. We must show that
$$
\gamma_{bc} - \gamma_{a+b,c} + \gamma_{a,b+c} - \gamma_{ab} = 0\qquad\text{for all $0\le a,b,c< v$}.
$$
But this follows easily using that $\gamma_{\imath\jmath} = \sum_{k=\jmath}^{\imath+\jmath-1} \gamma_{1k} - \sum_{k=1}^{\imath-1} \gamma_{1k}$ for all $\imath,\jmath\in \mathds{N}$.
\end{proof}

\begin{remark}\label{rem 5.4} Lemma~\ref{lema 5.3} implies that each $\sum \gamma_{\imath\jmath}  g^{\imath}\ot g^{\jmath} \in \ker\bigl(\wh{d}^{003}_{\mathrm{v}}\bigr)$ is uniquely determined by $\gamma_{11}, \dots,\gamma_{1,v-1}$. For example, for each $1\le b<v$, the element $f_b(\gamma)\coloneqq \sum \Lambda(\gamma,b)_{\imath\jmath}  g^{\imath}\ot g^{\jmath}$, where
\begin{equation}\label{bala1}
\Lambda(\gamma,b)_{\imath\jmath}\coloneqq \begin{cases} \gamma &\text{if $\imath\le b$ and $b-\imath<\jmath\le b$,}\\ -\gamma &\text{if $\imath>b$ and $b<\jmath\le v-\imath+b$,}\\ 0 &\text{otherwise,}
\end{cases}
\end{equation}
is the unique  $\sum \gamma_{\imath\jmath} g^{\imath}\ot g^{\jmath}\in \ker\bigl(\wh{d}^{003}_{\mathrm{v}}\bigr)$ with $\gamma_{1b} = \gamma$ and $\gamma_{1\jmath}=0$ for $\jmath\ne b$. Note that
$$
\sum \gamma_{\imath\jmath}  g^{\imath}\ot g^{\jmath} = \sum_{b=1}^{v-1} f_b(\gamma_{1b})\qquad\text{for each $\sum \gamma_{\imath\jmath}  g^{\imath}\ot g^{\jmath} \in \ker\bigl(\wh{d}^{003}_{\mathrm{v}}\bigr)$.}
$$
Thus $\{f_b(\gamma): 1\le b<v\text{ and } \gamma\in\Gamma\}$ generate $\ker\bigl(\wh{d}^{003}_{\mathrm{v}}\bigr)$.
\end{remark}

\begin{remark}\label{rem 5.5} A direct computation shows that
$$
\wh{d}_{\mathrm{v}}^{002}(\gamma  g^{\imath}) = -\sum_{\substack{\jmath=1 \\ \jmath\ne \imath}}^{v-1} \gamma  (g^{\imath}\ot g^{\jmath}+g^{\jmath}\ot g^{\imath}) -2\gamma g^{\imath}\ot g^{\imath} + \sum_{\substack{a,b = 1\\ a+b\equiv \imath\pmod{v}}}^{v-1} \gamma   g^a\ot g^b.
$$
Thus, by Remark~\ref{rem 5.4},
$$
\wh{d}_{\mathrm{v}}^{002}(\gamma  g^{\imath}) = \begin{cases} - 2 f_1(\gamma) - f_2(\gamma) -\cdots- f_{v-1}(\gamma) & \text{if $\imath = 1$,}\\  f_{\imath-1}(\gamma) - f_{\imath}(\gamma) & \text{if $\imath \ne 1$.}
\end{cases}
$$
Consequently, since the $f_{\imath}(\gamma)$'s generates $\ker\bigl(\wh{d}_{\mathrm{v}}^{003}\bigr)$, we have $\ker\bigl(\wh{d}_{\mathrm{v}}^{003}\bigr)/\ima\bigl(\wh{d}_{\mathrm{v}}^{002}\bigr) = \Gamma/v\Gamma$.
\end{remark}

\begin{lemma}\label{Calculo de wh{d}_{mathrm{h}1}^{012}(f_1(gamma))} Let $\gamma\in \Gamma$. If $u=v$, then $\wh{d}_{\mathrm{h}1}^{012}(f_1(\gamma)) = 0$. Otherwise
\begin{equation}\label{zazaza}
\begin{aligned}
\wh{d}_{\mathrm{h}1}^{012}(f_1(\gamma)) & = - \gamma g\ot g - \sum_{k=1}^{u'} \gamma g\ot g^{kt} - \sum_{k=u'+1}^{2u'-1} \gamma g\ot g^{kt+1} - \sum_{h=2}^{t-1} \sum_{k=hu'}^{(h+1)u'-1} \gamma g\ot g^{kt+h}\\
& + \sum_{\imath=2}^{v-1}\sum_{\jmath=1}^{v-1} \bigl(\Lambda(\gamma,1)_{(1-u)\imath,(1-u)\jmath}-\Lambda(\gamma,1)_{\imath\jmath}\bigr) g^{\imath} \ot g^{\jmath}.
\end{aligned}
\end{equation}
\end{lemma}

\begin{proof} By definition
$$
\wh{d}_{\mathrm{h}1}^{012}(f_1(\gamma)) = \sum \bigl(\Lambda(\gamma,1)_{(1-u)\imath,(1-u)\jmath}-\Lambda(\gamma,1)_{\imath\jmath} \bigr) g^{\imath}\ot g^{\jmath}.
$$
We will use~\eqref{bala1} in order to compute $\Lambda(\gamma,1)_{1-u,(1-u)\jmath}$. In order to carry out this task, for each $0<\jmath<v$ we need to find $k$ such that $0\le kv+(1-u)\jmath < v$. But this happens if and only if $(k-1)v<\jmath(u-1)\le kv$, and it is evident that such a $k$ there exists and it is unique. Moreover, $1\le k<u$ and $\jmath(u-1)\ne kv$. In fact, if $k\ge u$, then $kv+(1-u)\jmath > uv+(1-u)v\ge v$, while if $\jmath(u-1) = kv$, then $v\mid \jmath$, because $\gcd(u-1,v)=1$. By equality~\eqref{bala1}, for all $\jmath, k$, such that $0<\jmath<v$ and $0< kv+(1-u)\jmath < v$, we have
$$
\Lambda(\gamma,1)_{v-u+1,kv+\jmath(1-u)} = \begin{cases}\phantom{-} \gamma &\text{if $u=v$ and $0<kv+\jmath(1-u)\le 1$,}\\ -\gamma & \text{if $u<v$ and $1< kv+\jmath(1-u)\le u$,}\\ \phantom{-} 0 & \text{otherwise.} \end{cases}
$$
Consequently, $\Lambda(\gamma,1)_{v-u+1,kv+\jmath(1-u)} = \gamma$ if and only if $u=v$ and $k=\jmath = 1$. From this and Remark~\ref{rem 5.4}~it~fol\-lows easily that if $u=v$, then $\wh{d}_{\mathrm{h}1}^{012}(f_1(\gamma)) = 0$. Assume now that $u<v$. We next purpose~is~to~de\-ter\-mine when $\Lambda(\gamma,1)_{v-u+1,kv+\jmath(1-u)} = -\gamma$. Note that
$$
1< kv+\jmath(1-u)\le u \Leftrightarrow kv-u\le \jmath(u-1)<kv-1 \Leftrightarrow kt+\frac{kt-u}{u-1}\le \jmath < kt+\frac{kt-u}{u-1}+1.
$$
Thus,
$$
\Lambda(\gamma,1)_{v-u+1,kv+\jmath(1-u)} = -\gamma \quad\text{if and only if}\quad \jmath = kt+\left\lceil \frac{kt-u}{u-1}\right\rceil.
$$
Write $k = hu'+l$, where $h=0$ and $1\le l\le u'$, or $h=1$ and $1\le l< u'$, or $1<h<t$ and $0\le l< u'$. Then
$$
\frac{kt-u}{u-1} = \frac{(hu'+l)t-u}{u-1} = \frac{hu+lt-u}{u-1} = \frac{(h-1)(u-1)-h+lt+h-1}{u-1} = h-1 + \frac{lt+h-1}{u-1}.
$$
We claim that $\bigl\lceil \frac{kt-u}{u-1}\bigr\rceil = h$. To check this we must prove that $0<lt+h-1<l$. Assume first that $h=0$. Then $lt-1>0$, because $t>1$ and $l>0$; while $lt-1<u't-1=u-1$, because $0<l<u'$. Assume now that $1\le h <t$. Then $lt+h-1>0$, because $l>0$ or $h>1$; while $lt+h-1\le (u'-1)t+h-1 = u-t+h-1<u-1$. Summarizing, for all $\jmath, k$, such that $0<\jmath<v$ and $0< kv+(1-u)\jmath < v$, we have
$$
\Lambda(\gamma,1)_{v-u+1,kv+\jmath(1-u)} = -\gamma \quad\text{if and only if}\quad \jmath = \begin{cases} kt & \text{if $1\le k\le u'$,}\\ kt+1 & \text{if $u'+1\le k< 2u'$,}\\ kt+h & \text{if $hu'\le k< (h+1)u'$ with $1<h<t$.}\end{cases}
$$
Using this fact it is easy to see that equality~\eqref{zazaza} holds.
\end{proof}

\begin{remark}\label{fundam} Let $B,Z\subseteq \Gamma(2)^{00}\oplus \Gamma(1)^{01}\oplus \Gamma(1)^{10}$ be the $2$-coboundaries and the $2$-cocycles of $\mathcal{X}(\mathcal{A},\Gamma)$ respectively, and let $Z'\subseteq Z$ be the subgroup of cocycles $z = \bigl(\sum \gamma_{\imath\jmath} g^{\imath}\ot g^{\jmath},\sum \gamma_{\imath} g^{\imath},\sum \gamma'_{\imath} g^{\imath}\bigr)$, such that $\sum \gamma_{\imath\jmath} g^{\imath}\ot g^{\jmath} = f_1(\gamma)$ for some $\gamma\in \Gamma$. By Remark~\ref{rem 5.5} we know that $\Ho^2(\mathcal{X}(\mathcal{A},\Gamma)) = Z'/B\cap Z'$. Moreover, since $f_1(\gamma)\in\ker\bigl(\wh{d}_{\mathrm{v}}^{003}\bigr)$, a triple $z = \bigl(f_1(\gamma),\sum \gamma_{\imath} g^{\imath},\sum \gamma'_{\imath} g^{\imath}\bigr)$ is in $Z'$ if and only if
\begin{enumerate}

\item $\wh{d}^{201}_{\mathrm{h}0}\bigl(\sum \gamma'_{\imath} g^{\imath}\bigr) = 0$, $\wh{d}^{111}_{\mathrm{h}1}\bigl(\sum \gamma'_{\imath} g^{\imath}\bigr) = 0$ and $\wh{d}^{102}_{\mathrm{v}}\bigl(\sum \gamma'_{\imath} g^{\imath}\bigr) = 0$,

\item $\wh{d}^{021}_{\mathrm{h}2}\bigl(\sum \gamma'_{\imath} g^{\imath}\bigr) = - \wh{d}^{021}_{\mathrm{h}1}\bigl(\sum \gamma_{\imath} g^{\imath}\bigr)$,

\item $\wh{d}^{012}_{\mathrm{v}}\bigl(\sum \gamma_{\imath} g^{\imath}\bigr) = - \wh{d}^{012}_{\mathrm{h}1}\bigl(f_1(\gamma)\bigr)$.

\end{enumerate}
Clearly the first condition is satisfied if and only if $u\gamma'_{\imath} = 0$ for all $\imath$, $\gamma'_{\imath} = \gamma'_{\imath-\imath u}$ for all $\imath$, and $\gamma'_{\imath+\jmath} = \gamma'_{\imath}+\gamma'_{\jmath}$ for all $\imath,\jmath$. But this happens if and only if
\begin{equation}\label{pepitito}
\gamma'_{\imath} = \imath\gamma'_1\text{ for all $\imath$}\quad\text{and}\quad u\gamma'_1 = 0.
\end{equation}
On the other hand, item~(3) says that
$$
\sum (\gamma_{\imath} + \gamma_{\jmath}-\gamma_{\imath+\jmath})g^{\imath}\ot g^{\jmath} = \wh{d}^{012}_{\mathrm{v}}\Bigl(\sum \gamma_{\imath} g^{\imath}\Bigr) = - \wh{d}^{012}_{\mathrm{h}1}\bigl(f_1(\gamma)\bigr),
$$
which, by Lemma~\ref{Calculo de wh{d}_{mathrm{h}1}^{012}(f_1(gamma))}, implies that

\begin{enumerate}[itemsep=0.7ex, topsep=1.0ex, label=(\arabic*), start=4, leftmargin=0.9cm]

\item If $t = 1$ (or, equivalently, $u = v$), then $\gamma_{\imath} = \imath \gamma_1$. Moreover $v\gamma_1 = \gamma_v = 0$.

\item If $1 < t = u$ (or, equivalently, $u'=1$), then
\begin{equation}\label{maron}
\gamma_{k t+ l} = \begin{cases} (kt+l)\gamma_1 - (k+1)\gamma &\text{if $k=0$ and $2 \le l \le t$,}\\
(kt+l)\gamma_1 - (k+1)\gamma &\text{if $k=1$ and $1\le l \le t+2$,}\\
(kt+l)\gamma_1 - (k+1)\gamma &\text{if $2\le k<t-1$ and $k<l \le t+k+1$,}
\end{cases}
\end{equation}
and $v\gamma_1 - u\gamma = \gamma_v = 0$.

\item If $1<t<u$ (or equivalently, $1<u'<u$), then

\begin{equation}\label{maron'}
\qquad\quad\gamma_{k t+ l} = \begin{cases} (kt+l)\gamma_1 - (k+1)\gamma &\text{if $k=0$ and $2 \le l \le t$,}\\
(kt+l)\gamma_1 - (k+1)\gamma &\text{if $1\le k< u'$ and $1\le l \le t$,}\\
(kt+l)\gamma_1 - (k+1)\gamma &\text{if $k=u'$ and $1\le l \le t+1$,}\\
(kt+l)\gamma_1 - (k+1)\gamma &\text{if $u'<k\le 2u'-2$ and $2\le l \le t+1$,}\\
(kt+l)\gamma_1 - (k+1)\gamma &\text{if $2\le h<t$, $k=hu'-1$ and $h\le l \le t+h$,}\\
(kt+l)\gamma_1 - (k+1)\gamma &\text{if $2\le h < t$, $hu'\le k\le (h+1)u'-2$ and $h< l \le t+h$,}
\end{cases}
\end{equation}
and $v\gamma_1 - u\gamma = \gamma_v = 0$ (note that, if $u'=2$, then the fourth line in~\eqref{maron'} is empty; while, if $t=2$, then the last two lines are empty).
\end{enumerate}
Conversely under these conditions, item~(3) holds. Summarizing, items~(1) and~(3) are satisfied if and only if equality~\eqref{pepitito} and conditions~(4), (5) or~(6) are fulfilled, depending on the case. Finally, by the definition of $\wh{d}^{021}_{\mathrm{h}2}$ and equality~\eqref{pepitito},
\begin{equation}\label{etiq1}
\!\wh{d}^{021}_{\mathrm{h}2}\Bigl(\sum \gamma'_{\imath} g^{\imath}\Bigr) = \sum_{\imath} \left(\!-\gamma'_{\imath} + u'\sum_{s=1}^{t-1} s \gamma'_{\imath-\imath us}\!\right) g^{\imath} = \sum_{\imath} \left(\!u'\binom{t}{2}-1\right) \gamma'_{\imath} g^{\imath} = \begin{cases} - \sum \gamma'_{\imath}  g^{\imath} &\text{if $t\ne 2$,}\\ \frac{u-2}{2}\sum \gamma'_{\imath}g^{\imath} &\text{if $t=2$.} \end{cases}
\end{equation}
\end{remark}

\begin{lemma}\label{lema previo} Let $\sum \gamma_{\imath} g^{\imath}\in \Gamma(1)^{01}$ and $\gamma\in\Gamma$. If $1\!<\!u\!<\!v\!=\!u^2$ and condition~(5) holds, or $1\!<\!u\!<\!v\!<\!u^2$ and condition~(6) holds, then
\begin{equation}\label{ccc111'}
\wh{d}_{\mathrm{h}1}^{021}\Bigl(\sum \gamma_{\imath} g^{\imath}\Bigr) = - \sum_{\jmath=0}^{\eta-\nu-1} \sum_{\{\imath:v(\imath)=\jmath\}}
\left(\imath+ u'p^{2\jmath}\binom{t(\jmath)}{2}\right)(t\gamma_1-\gamma)g^{\imath} - \sum_{\{\imath:t\mid \imath\}} \imath \bigl(t\gamma_1-\gamma\bigr)g^{\imath},
\end{equation}
where $v(\imath)\coloneqq \max\{l\ge 0:p^l \mid \imath\}$ and $t(\jmath)\coloneqq t/p^{\jmath}$.
\end{lemma}

\begin{proof} Assume first that $v<u^2$. By Remark~\ref{fundam} we know that $v\gamma_1-u\gamma=0$ and equality~\eqref{maron'} is satisfied. A direct computation shows that this equality can be written as
\begin{equation}\label{maron''''}
\gamma_{k t+l} = \begin{cases} (k t+l)\gamma_1 - (k+1)\gamma &\text{if $k = 0$ and $2\le l<t$,}\\
(k t+l)\gamma_1 - k\gamma &\text{if $1\le k\le u'$ and $l= 0$,}\\
(k t+l)\gamma_1 - (k+1)\gamma &\text{if $1\le k\le u'$ and $1\le l<t$,}\\
(k t+l)\gamma_1 - k\gamma &\text{if $u'<k< 2u'$ and $0\le l\le 1$,}\\
(k t+l)\gamma_1 - (k+1)\gamma &\text{if $u'<k< 2u'$ and $2\le l<t$,}\\
(kt+l)\gamma_1 - k\gamma &\text{if $2\le h < t$, $hu'\le k<(h+1)u'$ and $0\le l\le h$,}\\
(kt+l)\gamma_1 - (k+1)\gamma &\text{if $2\le h < t$, $hu'\le k<(h+1)u'$ and $h<l<t$.}
\end{cases}
\end{equation}
Write $\imath \coloneqq p^{v(\imath)}\imath'$. Clearly
$$
\sum_{s=0}^{t-1} \gamma_{\imath-su\imath} = \sum_{s=0}^{t-1} \gamma_{\imath-s u p^{v(\imath)}\imath'} = \sum_{s=0}^{t-1} \gamma_{\imath+sup^{v(\imath)}} = p^{v(\imath)}\sum_{s=0}^{t(v(\imath))-1} \gamma_{\imath+sup^{v(\imath)}}.
$$	
Con\-se\-quently,
$$
\wh{d}_{\mathrm{h}1}^{021}\Bigl(\sum \gamma_{\imath} g^{\imath}\Bigr) = - \sum_{\imath} \biggl(\sum_{s=0}^{t-1} \gamma_{\imath-su\imath}\biggr) g^{\imath} = \sum_{\jmath=0}^{\eta-\nu-1}  p^{\jmath} \sum_{\{\imath:v(\imath)=\jmath\}} \sum_{s=0}^{t(\jmath)-1} \gamma_{\imath+sup^{\jmath}}g^{\imath} + \sum_{\{\imath:t\mid\imath\}} t\gamma_{\imath} g^{\imath}.
$$
By equality~\eqref{maron''''}, if $t\mid\imath$, then $t\gamma_{\imath} = t\imath \gamma_1 - \imath\gamma$. So, in order to finish the proof of equality~\eqref{ccc111'}, we only must check that
\begin{equation}\label{ccc22'}
\sum_{s=0}^{t(\jmath)-1} \gamma_{\imath+sup^{\jmath}} = \left(\frac{\imath}{p^{\jmath}} + u'p^{\jmath}\binom{t(\jmath)}{2}\right)(t\gamma_1-\gamma) \qquad\text{for all $\imath$ such that $v(\imath) = \jmath$}.
\end{equation}
We divided the proof of this in five cases. In the first four we use equality~\eqref{maron''''} and that $u=tu'$.

\smallskip
		
\noindent 1)\enspace If $\imath = 1$, then $\jmath\coloneqq v(\imath) = 0$, and so
$$
\qquad \sum_{s=0}^{t(\jmath)-1} \gamma_{\imath+sup^{\jmath}} = \sum_{s=0}^{t-1} \bigl(1+su\bigr)\gamma_1 - \left(u'+1 + \sum_{s=2}^{t-1} su'\right)\gamma = \left(1 + u'\binom{t}{2}\right) \bigl(t\gamma_1-\gamma\bigr).
$$

\smallskip
		
\noindent 2)\enspace If $1<\imath<t$, then
$$
\qquad \sum_{s=0}^{t(\jmath)-1} \gamma_{\imath+sup^{\jmath}} = \sum_{s=0}^{t(\jmath)-1} \bigl(\imath+sup^{\jmath}\bigr)\gamma_1 - \left(\sum_{s=0}^{\frac{\imath}{p^{\jmath}}-1} (su'p^{\jmath}+1) + \sum_{s=\imath}^{t(\jmath)-1} su'p^{\jmath}\right) \gamma = \left(\frac{\imath}{p^{\jmath}} + u'p^{\jmath}\binom{t(\jmath)}{2}\right) \bigl(t\gamma_1-\gamma\bigr).
$$

\smallskip

\noindent 3)\enspace If $t <\imath<u$, then $\imath = tq+\bar{\imath}$ with $0<q<u'$ and $0<\bar{\imath}<t$, which implies that
$$
\qquad \sum_{s=0}^{t(\jmath)-1} \gamma_{\imath+sup^{\jmath}} = \sum_{s=0}^{t(\jmath)-1} \gamma_{\bar{\imath}+(su'p^{\jmath}+q)t} = \sum_{s=0}^{t(\jmath)-1}(\imath+ su'p^{\jmath}t)\gamma_1 -  \left(\sum_{s=0}^{\frac{\bar{\imath}}{p^{\jmath}}-1} (su'p^{\jmath}+q+1) + \sum_{s=\frac{\bar{\imath}}{p^{\jmath}}}^{t(\jmath)-1}(su'p^{\jmath}+q)\right)\gamma.
$$
Thus
$$
\sum_{s=0}^{p-1} \gamma_{\imath+sup^{\jmath}} = \left(t\frac{\imath}{p^{\jmath}} + u'p^{\jmath}t\binom{t(\jmath)}{2}\right)\gamma_1 -\left(\frac{\bar{\imath}}{p^{\jmath}} + t(\jmath)q + u'p^{\jmath}\binom{t(\jmath)}{2}\right)\gamma = \left(\frac{\imath}{p^{\jmath}} + u'p^{\jmath}\binom{t(\jmath)}{2}\right)(t\gamma_1-\gamma).
$$
		
\smallskip

\noindent 4)\enspace If $\varsigma u <\imath<(\varsigma\!+\!1)u$, where $0< \varsigma<p^{\jmath}$, then $\imath = tq+\bar{\imath}$ with $\varsigma u'\le q<(\varsigma\!+\!1)u'$ and $0<\bar{\imath}<t$. So,
$$
\sum_{s=0}^{t(\jmath)-1} \gamma_{\imath+sup^{\jmath}} = \sum_{s=0}^{t(\jmath)-1} \gamma_{\bar{\imath}+(su'p^{\jmath}+q)t} = \sum_{s=0}^{t(\jmath)-1}(\imath+ su'p^{\jmath}t)\gamma_1 -  \left(\sum_{s=0}^{\left\lceil\frac{\bar{\imath}-\varsigma}{p^{\jmath}}\right\rceil-1} (su'p^{\jmath}+q+1) + \sum_{s={\left\lceil\frac{\bar{\imath}-\varsigma}{p^{\jmath}}\right\rceil}}^{t(\jmath)-1}(su'p^{\jmath}+q)\right)\gamma.
$$
Since $\left\lceil\frac{\bar{\imath}-\varsigma}{p^{\jmath}}\right\rceil=\frac{\bar{\imath}}{p^{\jmath}}$, this implies that
$$
\sum_{s=0}^{t(\jmath)-1} \gamma_{\imath+sup^{\jmath}} = \left(t\frac{\imath}{p^{\jmath}} + u'p^{\jmath}t\binom{t(\jmath)}{2}\right)\gamma_1 -\left(\frac{\bar{\imath}}{p^{\jmath}} + t(\jmath)q + u'p^{\jmath}\binom{t(\jmath)}{2}\right)\gamma = \left(\frac{\imath}{p^{\jmath}} + u'p^{\jmath}\binom{t(\jmath)}{2}\right)(t\gamma_1-\gamma).
$$

\smallskip

\noindent 5)\enspace If $p^{\jmath}u<\imath$, then by the previous cases we have
$$
\sum_{s=0}^{t(\jmath)-1} \gamma_{\imath+sup^{\jmath}} = \sum_{s=0}^{t(\jmath)-1} \gamma_{\bar{\imath}+sup^{\jmath}} = \left(\frac{\bar{\imath}}{p^{\jmath}} + u'p^{\jmath}\binom{t(\jmath)}{2}\right)(t\gamma_1-\gamma) = \left(\frac{\imath}{p^{\jmath}} + u'p^{\jmath}\binom{t(\jmath)}{2}\right)(t\gamma_1-\gamma),
$$
where $1\le\bar{\imath}<p^{\jmath}u$ is the remainder of the integer division of $\imath$ by $p^{\jmath}u$ (the last equality follows from the fact that $v\gamma_1-u\gamma=0$).
		
\smallskip

Assume now that $v=u^2$. By Remark~\ref{fundam} we know that $v\gamma_1-u\gamma=0$ and equality~\eqref{maron} is satisfied. A direct computation shows that this equality can be written as
\begin{equation}\label{maron''}
\gamma_{kt+l} = \begin{cases} (k t+l)\gamma_1 - (k+1)\gamma &\text{if $k=0$ and $2\le l<t$,}\\
(kt+l)\gamma_1 - k\gamma &\text{if $k=1$ and $l=0$,}\\
(kt+l)\gamma_1 - (k+1)\gamma &\text{if $k=1$ and $1\le l<t$,}\\
(kt+l)\gamma_1 - k\gamma &\text{if $1<k<t$ and $0\le l\le k$,}\\
(kt+l)\gamma_1 - (k+1)\gamma &\text{if $1<k<t$ and $k<l<t$.}
\end{cases}
\end{equation}
In the case $v=u^2$ the proof of equality~\eqref{ccc111'} follows the same pattern than in the case $v<u^2$, but using equality~\eqref{maron''} instead of~\eqref{maron''''}. We leave the details to the reader.
\end{proof}

\begin{lemma}\label{lema previo1} Let $\sum \gamma_{\imath} g^{\imath}\in \Gamma(1)^{01}$, $\sum \gamma'_{\imath} g^{\imath}\in \Gamma(1)^{10}$ and $\gamma\in\Gamma$. Assume the hypothesis of Lemma~\ref{lema previo} holds and that $u>2$. Then equality~\eqref{pepitito} and condition~(2) are satisfied if and only if $\gamma'_{\imath} = -\imath(t\gamma_1-\gamma)$ for all $\imath$.
\end{lemma}

\begin{proof} By Lemma~\ref{lema previo} we have
\begin{equation}\label{ccc11'}
\wh{d}_{\mathrm{h}1}^{021}\Bigl(\sum \gamma_{\imath} g^{\imath}\Bigr) = - \sum_{\jmath=0}^{\eta-\nu-1} \sum_{\{\imath:v(\imath)=\jmath\}}
\left(\imath+ u'p^{2\jmath}\binom{t(\jmath)}{2}\right)(t\gamma_1-\gamma)g^{\imath} - \sum_{\{\imath:t\mid \imath\}} \imath \bigl(t\gamma_1-\gamma\bigr)g^{\imath}.
\end{equation}
Assume first that $p$ is odd. Then $u\mid u'p^{2\jmath}\binom{t(\jmath)}{2}$ for all $0\le \jmath <\eta-\nu$, and thus $u'p^{2\jmath}\binom{t(\jmath)}{2}(t\gamma_1-\gamma) = 0$, since $u(t\gamma_1-\gamma) = v\gamma_1-u\gamma =0$. Consequently, by equalities~\eqref{etiq1} and~\eqref{ccc11'}, condition~(2) holds if and only if $\gamma'_{\imath} = -\imath (t\gamma_1-\gamma)$ for all $\imath$ (note that these $\gamma'_{\imath}$'s satisfy condition~\eqref{pepitito}). Assume now that $p=2$ and $\nu>1$. Since $2\mid t$, we have $u'\binom{t}{2} = \frac{u}{2}(t-1)\equiv \frac{u}{2} \pmod{u}$. Consequently, if condition~(2) is true, then
\begin{equation}\label{ec11}
\left(\frac{u}{2}-1\right)\gamma'_1 = \left(1 + u'\binom{t}{2}\right)(t\gamma_1-\gamma) = \left(1 + \frac{u}{2}\right)(t\gamma_1-\gamma).
\end{equation}
Since $u\gamma'_1 = u(t\gamma_1-\gamma) = 0$, this implies that $-2\gamma'_1 = 2(t\gamma_1-\gamma)$, and so $-\frac{u}{2}\gamma'_1 = \frac{u}{2}(t\gamma_1-\gamma)$, because $4\mid u$. Adding this equality to~\eqref{ec11}, we obtain that $-\gamma'_1 = (1+u) (t\gamma_1-\gamma) = t\gamma_1-\gamma$. By condition~\eqref{pepitito} this implies that $\gamma'_{\imath} = -\imath(t\gamma_1-\gamma)$ for all $\imath$ (note that these $\gamma'_{\imath}$'s satisfy condition~\eqref{pepitito}). Conversely assume that $\gamma'_{\imath} = -\imath(t\gamma_1-\gamma)$ for all $\imath$. By equalities~\eqref{etiq1} and~\eqref{ccc11'}, in order to prove that condition~(2) is satisfied, we must check that
\begin{equation}\label{ec22}
\left(\frac{u}{2}-1\right)\gamma'_{\imath} = \left(\imath+ u'2^{2\jmath}\binom{t(\jmath)}{2}\right)(t\gamma_1-\gamma),
\end{equation}
for all $\imath$ such that $\jmath\coloneqq v(\imath)\in \{0,\dots,\eta-\nu-1\}$. If $\jmath>0$, then
$$
u'2^{2\jmath}\binom{t(\jmath)}{2}(t\gamma_1-\gamma) = u'2^{\jmath}\frac{t}{2}(t(\jmath)-1)(t\gamma_1-\gamma) = 2^{\jmath-1}u(t(\jmath)-1) (t\gamma_1-\gamma)=0,
$$
and so, since $2\mid \imath$, we have
$$
\left(\frac{u}{2}-1\right)\gamma'_{\imath} = \left(1-\frac{u}{2}\right)\imath(t\gamma_1-\gamma) = \imath(t\gamma_1-\gamma) = \left(\imath+ u'2^{2\jmath}\binom{t(\jmath)}{2}\right)(t\gamma_1-\gamma),
$$
as desired. Assume then that $\jmath = 0$. Hence
$$
u'2^{2\jmath}\binom{t(\jmath)}{2} = u'\binom{t}{2} = \frac{u}{2}(t-1)\equiv  - \frac{u}{2}\pmod{u},
$$
where the last equality holds since $2\mid t$. Since, moreover $\frac{u}{2}(\imath-1)\gamma'_1 = 0$, we have $\frac{u}{2}\gamma'_{\imath} =  \frac{u}{2}\imath\gamma'_1  =  \frac{u}{2}\gamma'_1$, and so
$$
\left(\imath+ u'2^{2\jmath}\binom{t(\jmath)}{2}\right)(t\gamma_1-\gamma) = \left(\imath-\frac{u}{2}\right)(t\gamma_1-\gamma) = -\gamma'_{\imath} + \frac{u}{2}\gamma'_1 = \left(-1+\frac{u}{2}\right)\gamma'_{\imath},
$$
which finishes the proof.
\end{proof}

Let $\mathcal{A}$ be as at the beginning of this Section, let $\Gamma$ be an additive abelian group and let $B$ and $Z$ be the groups of $2$-coboundaries and $2$-cocycles of $\mathcal{X}(\mathcal{A},\Gamma)$, respectively. Recall that $\Gamma_{\!r}\coloneqq \{\gamma\in \Gamma : r\gamma=0\}$, for each natural number $r$.

\smallskip

In the following result we set $z(\gamma_1,\gamma)\coloneqq \left(f_1(\gamma),\sum \imath \gamma_1 g^{\imath},-\sum \imath\gamma_1 g^{\imath} \right)$, where $\gamma,\gamma_1\in \Gamma$.

\begin{theorem}\label{caso t=1} If $u=v$, then
\begin{equation}\label{ec5}
\Ho_N^2(\mathcal{A},\Gamma) = \wh{\Ho}_N^2(\mathcal{A},\Gamma) = \Ho^2(\mathcal{X}(\mathcal{A},\Gamma))\simeq \Gamma_{\!v}\oplus \frac{\Gamma}{v\Gamma}.
\end{equation}
Moreover, the set $\ov{Z}\coloneqq \left\{\left(z(\gamma_1,\gamma)\right): \text{ $\gamma\in\Gamma$ and $\gamma_1\in \Gamma_{\!u}$}\right\}$, is a subgroup of $Z$, that applies surjectively on $\Ho^2(\mathcal{X}(\mathcal{A},\Gamma))$ and $B\cap \ov{Z} = \{(f_1(\gamma),0,0):\gamma\in v\Gamma\}$. Finally the map $\Theta\colon \ov{Z}\to \Gamma \oplus \Gamma_{\!u}$, defined by $\Theta\left(z(\gamma_1,\gamma) \right)\coloneqq (\gamma_1,\gamma)$ is an isomorphism that induces the isomorphism in~\eqref{ec5}.
\end{theorem}

\begin{proof} By Remarks~\ref{son iguales2} and~\ref{util para'} we have $\Ho_N^2(\mathcal{A},\Gamma) = \wh{\Ho}_N^2(\mathcal{A},\Gamma) = \Ho^2(\mathcal{X}(\mathcal{A},\Gamma))$. Let $Z'$ be as in Remark~\ref{fundam}. Thus
$$
Z'=\left\{\left(f_1(\gamma),\sum \gamma_{\imath} g^{\imath},\sum \gamma'_{\imath} g^{\imath}\right):\text{$v\gamma_1=u\gamma'_1=0$ and condition~(2) is satisfied}\right\}.
$$
Since $t=1$, by the definition of $\wh{d}_{\mathrm{h}1}^{021}$ and equality~\eqref{etiq1}, condition~(2) holds if and only if $\gamma_1 = \gamma'_1$. Hence $Z' =\ov{Z}$. Clearly the map $\Theta\colon \ov{Z}\to \Gamma_{\!v}\oplus \Gamma$, defined by $\Theta\left(z(\gamma_1,\gamma) \right)\coloneqq (\gamma_1,\gamma)$ is an isomorphism. We now compute
$$
B\cap \ov{Z} = \left\{\left(\wh{d}_{\mathrm{v}}^{002}(x),\wh{d}_{\mathrm{h}1}^{011}(x),0\right):x\in \Gamma(1)^{00}\text{ and } \wh{d}_{\mathrm{v}}^{002}(x) = f_1(\gamma) \text{ for some }\gamma\in\Gamma\right\}.
$$
Write $x = \sum \gamma''_{\imath} g^{\imath}$. By the definition of $\wh{d}_{\mathrm{v}}^{002}$, we have
$$
\wh{d}_{\mathrm{v}}^{002}(x) = \sum_{\jmath = 1}^{v-2} (\gamma''_{\jmath+1}-\gamma''_1-\gamma''_{\jmath}) g\ot g^{\jmath}  - (\gamma''_1+\gamma''_{v-1}) g\ot g^{v-1} + \sum_{\imath = 2}^{v-1}\sum_{\jmath = 1}^{v-1} (\gamma''_{\imath+\jmath}-\gamma''_{\jmath}-\gamma''_{\imath}) g^{\imath}\ot g^{\jmath}.
$$
So, by Remark~\ref{rem 5.4}, we get that $\wh{d}_{\mathrm{v}}^{002}(x) = f_1(\gamma)$ if and only if $\gamma_2''-2\gamma''_1=\gamma$, $\gamma''_{\imath+1}=\gamma''_1+\gamma''_{\imath}$ for $1<\imath <v-1$, and $\gamma''_{v-1}=-\gamma''_1$ (or, equivalently, if and only if $\gamma = -v\gamma''_1$ and $\gamma''_{\imath} = -(v-\imath) \gamma''_1$ for $1< \imath<v$). Moreover, since $u=v$, we have $\wh{d}_{\mathrm{h}1}^{011}(x) = 0$. Consequently $B\cap \ov{Z} = \left\{\left(f_1(-v\gamma''_1),0,0\right): \gamma''_1\in \Gamma\right\}$. Thus the map $\Theta$ induces an isomorphism $\Ho^2(\mathcal{X}(\mathcal{A},\Gamma))\simeq \Gamma_{\!v}\oplus \frac{\Gamma}{v\Gamma}$.
\end{proof}

Assume that we are under the hypothesis of Theorem~\ref{caso t=1}. For each $\gamma\!\in\!\Gamma$ and $\gamma_1\!\in\!\Gamma_v$, let $\xi_{\gamma}^1\colon \ov{M}(2)\to \Gamma$ and $\xi_{\gamma_1}^2\colon \ov{D}\ot \ov{M}(1)\to \Gamma$ be the maps defined by
\begin{equation*}
\xi_{\gamma}^1([g^{\imath_1}\ot g^{\imath_2}])\coloneqq \begin{cases}\phantom{-} \gamma &\text{if $\imath_1=\imath_2=1$,}\\ -\gamma &\text{if $\imath_1,\imath_2\ge 2$ and $\imath_1+\imath_2\le v+1$,}\\ \phantom{-} 0 &\text{otherwise,}\end{cases} \quad\text{and}\quad \xi_{\gamma_1}^2(g^{\imath_1}\ot g^{\imath_2} )\coloneqq \imath_1\imath_2\gamma_1.
\end{equation*}

\begin{proposition}\label{complemento 1} The map $(\xi_{\gamma}^1,\xi_{\gamma_1}^2) \colon \ov{M}(2)\oplus (\ov{D}\ot \ov{M}(1))\to \Gamma$ is a $2$-cocycle of $\wh{C}_N^*(\mathcal{A},\Gamma)$. Moreover, each $2$-cocycle of $\wh{C}_N^*(\mathcal{A},\Gamma)$ is cohomologous to a $(\xi_{\gamma}^1,\xi_{\gamma_1}^2)$ and two $2$-cocycles $(\xi_{\gamma}^1,\xi_{\gamma_1}^2)$ and $(\xi_{\gamma'}^1,\xi_{\gamma'_1}^2)$ are cohomologous if and only if $\gamma'_1 = \gamma_1$ and $v\gamma'=v\gamma$.
\end{proposition}

\begin{proof} By Remark~\ref{mascalculos}, Theorem~\ref{caso t=1} and the discussion above Remark~\ref{pero}, it suffices to check that
$$
\bigl(\xi_{\gamma}^1,\xi_{\gamma_1}^2\bigr) = z(\gamma_1,\gamma)[\wh{\varphi}_2],
$$
where $z(\gamma_1,\gamma)$ is as in the statement of Theorem~\ref{caso t=1} and $[\wh{\varphi}_2]$ is as in~\eqref{matriz}. But this follows by that theorem, equality~\eqref{bala1} and Proposition~\ref{funciones de arriba} with $t=1$.
\end{proof}

\begin{proof}[Proof of Theorem~A] This follows from Remark~\ref{son iguales2}, Proposition~\ref{complemento 1} and \cite{LV}*{Theorem~5.8}.
\end{proof}

In the following result for each $\gamma,\gamma_1\in \Gamma$, we set $z(\gamma_1,\gamma)\coloneqq \left(f_1(\gamma),\sum \gamma_{\imath} g^{\imath},-\sum \imath (t\gamma_1-\gamma) g^{\imath} \right)$, where the $\gamma_i$'s with $i\ge 2$ are as in~\eqref{maron'} if $v<u^2$, and the $\gamma_i$'s with $i\ge 2$ are as in~\eqref{maron} if $v=u^2$.

\begin{theorem}\label{caso 1<t<u} If $2<u<v\le u^2$, then
\begin{equation}\label{ec3}
\Ho_N^2(\mathcal{A},\Gamma) = \wh{\Ho}_N^2(\mathcal{A},\Gamma) = \Ho^2(\mathcal{X}(\mathcal{A},\Gamma))\simeq \frac{\Gamma}{u\Gamma}\oplus \Gamma_{\!u}.
\end{equation}
Moreover the set $\ov{Z}\coloneqq \left\{z(\gamma_1,\gamma):v\gamma_1=u\gamma\right\}$ is a subgroup of $Z$ that applies surjectively on $\Ho^2(\mathcal{X}(\mathcal{A},\Gamma))$ and $B\cap \ov{Z} = \{z(u\gamma,v\gamma):\gamma\in \Gamma\}$. Finally the map $\Theta\colon \ov{Z}\to \Gamma \oplus \Gamma_{\!u}$, defined by $\Theta(z(\gamma_1,\gamma)) \coloneqq (\gamma_1,t\gamma_1-\gamma)$ is an isomorphism that induces the isomorphism in~\eqref{ec3}.
\end{theorem}

\begin{proof} Assume first that $v<u^2$. By Remarks~\ref{son iguales2} and~\ref{util para'} we have $\Ho_N^2(\mathcal{A},\Gamma) = \wh{\Ho}_N^2(\mathcal{A},\Gamma) = \Ho^2(\mathcal{X}(\mathcal{A},\Gamma))$. Let $Z'$ be as in Remark~\ref{fundam}. We have
$$
Z'=\left\{\left(f_1(\gamma),\sum \gamma_{\imath} g^{\imath},\sum \gamma'_{\imath} g^{\imath}\right):\text{$\gamma'_{\imath}\!=\!\imath\gamma'_1$, $u\gamma'_1\!=\!0$, $v\gamma_1\!=\!u\gamma$ and condition~(2) and equality~\eqref{maron'} hold}\right\}.
$$
By Lemma~\ref{lema previo1} we have $Z'=\ov{Z}$. Clearly, the map $\Theta\colon \ov{Z}\to \Gamma\oplus \Gamma_{\!u}$, given by $\Theta(z(\gamma_1,\gamma)) \coloneqq (\gamma_1,t\gamma_1-\gamma)$,
is an isomorphism. We now com\-pute
$$
B\cap \ov{Z} = \left\{\left(\wh{d}_{\mathrm{v}}^{002}(x),\wh{d}_{\mathrm{h}1}^{011}(x),0\right):x\in \Gamma(1)^{00}\text{ and } \wh{d}_{\mathrm{v}}^{002}(x) = f_1(\gamma) \text{ for some }\gamma\in\Gamma\right\}.
$$
Write $x = \sum \gamma''_{\imath} g^{\imath}$. By the definition of $\wh{d}_{\mathrm{v}}^{002}$, we have
$$
\wh{d}_{\mathrm{v}}^{002}(x) = \sum_{\jmath=1}^{v-2} (\gamma''_{\jmath+1}-\gamma''_1-\gamma''_{\jmath}) g\ot g^{\jmath} - (\gamma''_1+\gamma''_{v-1}) g\ot g^{v-1} + \sum_{\imath = 2}^{v-1}\sum_{\jmath=1}^{v-1}(\gamma''_{\imath+\jmath}- \gamma''_{\jmath}- \gamma''_{\imath}) g^{\imath}\ot g^{\jmath}.
$$
So, by Remark~\ref{rem 5.4}, we get that $\wh{d}_{\mathrm{v}}^{002}(x) = f_1(\gamma)$ if and only if $\gamma_2''-2\gamma''_1=\gamma$, $\gamma''_{\imath+1}=\gamma''_1+\gamma''_{\imath}$ for $1<\imath < v-1$, and $\gamma''_{v-1}=-\gamma''_1$ (or, equivalently, if and only if $\gamma = -v\gamma''_1$ and $\gamma''_{\imath} = (\imath-v)\gamma''_1$ for $1<\imath<v$). Hence,
$$
\wh{d}_{\mathrm{h}1}^{011}(x) = \wh{d}_{\mathrm{h}1}^{011}\Bigl(\sum \gamma''_{\imath} g^{\imath}\Bigr) = \sum \bigl(\gamma''_{(1-u)\imath}-\gamma''_{\imath}\bigr)g^{\imath} = -u\gamma''_1 g + \sum_{\imath\ge 2} \bigl(\gamma''_{(1-u)\imath}-\gamma''_{\imath}\bigr)g^{\imath}.
$$
Thus, $B\cap \ov{Z} = \{z(u\gamma,v\gamma):\gamma\in \Gamma\}$, and so the map $\Theta$ induces an isomorphism $\Ho^2(\mathcal{X}(\mathcal{A},\Gamma)) \simeq \frac{\Gamma}{u\Gamma}\oplus \Gamma_{\!u}$.

\smallskip

The case $v=u^2$ follows in the same way. The unique difference is that, in the characterization of $Z'$ we must use equality~\eqref{maron} instead of~\eqref{maron'}.
\end{proof}

Assume that we are under the hypothesis of Theorem~\ref{caso 1<t<u}. For each $\gamma,\gamma_1\in\Gamma$ such that $v\gamma_1=u\gamma$, let $\xi_{\gamma}^1\colon \ov{M}(2)\to \Gamma$ and $\xi_{\gamma_1,\gamma}^2\colon \ov{D}\ot \ov{M}(1)\to \Gamma$ be the maps defined by
\begin{align*}
& \xi_{\gamma}^1([g^{\imath_1}\ot g^{\imath_2}])\coloneqq \begin{cases} \phantom{-}\gamma &\text{if $\imath_1=\imath_2=1$,}\\ -\gamma &\text{if $\imath_1,\imath_2\ge 2$ and $\imath_1+\imath_2\le v+1$,}\\ \phantom{-} 0 &\text{otherwise,}\end{cases}
\shortintertext{and}
& \xi_{\gamma_1,\gamma}^2(g^{t\imath+\jmath}\ot g^{\imath_1})\coloneqq \imath_1\left(\imath  - u' \binom{j}{2}\right)(t\gamma_1-\gamma) + \sum_{l=0}^{\jmath-1}\gamma_{\imath_1-ul\imath_1},
\end{align*}
where $0\le \imath<u$, $0\le \jmath< t$ and the $\gamma_r$'s are as in equality~\eqref{maron'} if $v=u^2$ and they are in equality~\eqref{maron'} if $v<u^2$ (take into account that if $r<0$ or $r\ge v$, then to apply equalities~\eqref{maron} and~\eqref{maron'}, in order to compute explicitly the map $\gamma_r$ in function of $\gamma_1$ and $\gamma$, it is necessary to replace $r$ by the remainder of the integer division of $r$ by $v$).

\begin{proposition}\label{complemento 2} The map $(\xi_{\gamma}^1,\xi_{\gamma_1,\gamma}^2) \colon \ov{M}(2)\oplus (\ov{D}\ot \ov{M}(1))\to \Gamma$ is a $2$-cocycle of $\wh{C}^N_*(\mathcal{A},\Gamma)$. Moreover each $2$-cocycle of $\wh{C}^N_*(\mathcal{A},\Gamma)$ is cohomologous to a $(\xi_{\gamma}^1, \xi_{\gamma_1,\gamma}^2)$ and two $2$-cocycles $(\xi_{\gamma}^1,\xi_{\gamma_1,\gamma}^2)$ and $(\xi_{\gamma'}^1, \xi_{\gamma'_1,\gamma'}^2)$ are cohomologous if and only if $\gamma_1-\gamma'_1\in u \Gamma$ and $t(\gamma_1-\gamma'_1)=\gamma-\gamma'$.
\end{proposition}

\begin{proof} By Remark~\ref{mascalculos} and Theorem~\ref{caso 1<t<u} it suffices to check that $(\xi_{\gamma}^1,\xi_{\gamma_1,\gamma}^2) = z(\gamma_1,\gamma) [\wh{\varphi}_2]$, where $z(\gamma_1,\gamma)$ is as in the statement of Theorem~\ref{caso 1<t<u} and $[\wh{\varphi}_2]$ is as in~\eqref{matriz}. But this follows by that theorem, equality~\eqref{bala1}, the fact that $u(t\gamma_1-\gamma)=0$, and Proposition~\ref{funciones de arriba}.
\end{proof}

\begin{proof}[Proof of Theorem~B] This follows from Remark~\ref{son iguales2}, Proposition~\ref{complemento 2} and \cite{LV}*{Theorem~5.8}.
\end{proof}

In the following result we set $z(\gamma_1,\gamma'_1,\gamma)\coloneqq (f_1(\gamma),\sum \gamma_{\imath} g^{\imath},\sum \gamma'_{\imath}g^{\imath})$, where $\gamma,\gamma_1,\gamma'_1\in \Gamma$, the $\gamma_i$'s with $i\ge 2$ are as in~\eqref{maron} and $\gamma'_{\imath}=\imath \gamma'_1$ for $\imath\ge 2$.

\begin{theorem}\label{caso t=u=p} If $v=u^2=4$, then
\begin{equation}\label{ec4}
\Ho_N^2(\mathcal{A},\Gamma) = \wh{\Ho}_N^2(\mathcal{A},\Gamma) = \Ho^2(\mathcal{X}(\mathcal{A},\Gamma))\simeq \frac{\Gamma}{2\Gamma} \oplus \Gamma_{\!2}\oplus\Gamma_{\!2}.
\end{equation}
Moreover, the subgroup $\ov{Z} \coloneqq \left\{z(\gamma_1,\gamma'_1,\gamma): \text{$2\gamma'_1 = 0$ and $4\gamma_1= 2\gamma$} \right\}$ of $Z$ that surjectively on $\Ho^2(\mathcal{X}(\mathcal{A},\Gamma))$, $\ov{Z}\cap B = \{z(2\gamma,0,4\gamma):\gamma\in \Gamma\}$ and the map $\Theta\colon \ov{Z}\to \Gamma \oplus \Gamma_{\!2} \oplus \Gamma_{\!2}$, defined by $\Theta(z(\gamma_1,\gamma'_1,\gamma))\coloneqq (\gamma_1,2\gamma_1+\gamma,\gamma'_1)$, is an isomorphism that induces the isomorphism in~\eqref{ec4}.
\end{theorem}

\begin{proof} By Remarks~\ref{son iguales2} and~\ref{util para'} we have $\Ho_N^2(\mathcal{A},\Gamma) = \wh{\Ho}_N^2(\mathcal{A},\Gamma) = \Ho^2(\mathcal{X}(\mathcal{A},\Gamma))$. Let $Z'$ be as in Remark~\ref{fundam}. Then
$$
Z'=\left\{\left(f_1(\gamma),\sum \gamma_{\imath} g^{\imath},\sum \gamma'_{\imath} g^{\imath}\right):\text{$\gamma'_{\imath}\!=\!\imath\gamma'_1$, $u\gamma'_1\!=\!0$, $v\gamma_1\!=\!u\gamma$ and condition~(2) and equality~\eqref{maron} hold}\right\}.
$$
By the fact that $4\gamma_1-2\gamma = 0$ and equalities~\eqref{etiq1} and~\eqref{ccc111'}, we have
$$
\wh{d}_{\mathrm{h}1}^{021}\Bigl(\sum \gamma_{\jmath} g^{\jmath}\Bigr) = -2(2\gamma_1-\gamma) g - 2(2\gamma_1-\gamma) g^2 - 4(2\gamma_1-\gamma) g^3 =0 = \wh{d}_{\mathrm{h}2}^{021}\Bigl(\sum \gamma'_{\jmath} g^{\jmath}\Bigr),
$$
which shows in particular that condition~(2) is fulfilled. Hence $Z' = \ov{Z}$. Clearly the map $\Theta\colon \ov{Z}\to \Gamma\oplus \Gamma_{\!2}\oplus \Gamma_{\!2}$, defined by $\Theta(z(\gamma_1,\gamma'_1,\gamma))\coloneqq (\gamma_1,2\gamma_1+\gamma,\gamma'_1)$, is an isomorphism. We now compute the group $\ov{Z}\cap B$. Let $x \coloneqq \gamma''_1 g+\gamma''_2 g^2+\gamma''_3 g^3\in \Gamma(1)^{00}$. Arguing as above we get that $\wh{d}_{\mathrm{v}}^{002}(x) = f_1(\gamma)$ if and only if $\gamma''_2 = -2\gamma''_1$, $\gamma''_3 = -\gamma''_1$ and $\gamma=-4\gamma''_1$. Hence,
$$
\wh{d}_{\mathrm{h}1}^{011}\bigl(\gamma''_1 g + \gamma''_2 g^2+\gamma''_3 g^3\bigr) = (\gamma''_3-\gamma''_1)g+(\gamma''_2-\gamma''_2)g^2 + (\gamma''_1- \gamma''_3)g^3 = - 2\gamma''_1 g + 2\gamma''_1 g^3,
$$
and so $\ov{Z}\cap B = \{z(2\gamma,0,4\gamma):\gamma\in \Gamma\}$. Thus, $\Theta$ induces an isomorphism $\Ho^2(\mathcal{X}(\mathcal{A},\Gamma))\simeq \frac{\Gamma}{2\Gamma} \oplus \Gamma_{\!2}\oplus\Gamma_{\!2}$.
\end{proof}

Assume that we are under the hypothesis of Theorem~\ref{caso t=u=p}. For each $\gamma'_1\in \Gamma_2$ and $\gamma,\gamma_1\in\Gamma$ such that $4\gamma_1=2\gamma$, let $\xi_{\gamma}^1\colon \ov{M}(2)\to \Gamma$ be as above of Proposition~\ref{complemento 2} and $\xi_{\gamma_1,\gamma'_1,\gamma}^2\colon \ov{D}\ot \ov{M}(1)\to \Gamma$ be the map defined by
$$
\xi_{\gamma_1,\gamma'_1,\gamma}^2(g^{2\imath+\jmath}\ot g^{\imath_1})\coloneqq \sum_{l=0}^{\jmath-1} \gamma_{\imath_1-ul\imath_1} - \imath\imath_1\gamma'_1 = \begin{cases} 0 &\text{$\imath_1 = 0$ or $\imath=\jmath=0$,}\\ -\gamma'_1 &\text{if $\imath=1$, $\jmath=0$ and $\imath_1\in\{1,3\}$,}\\ 0 &\text{if $\imath=1$, $\jmath=0$ and $\imath_1 = 2$,}\\ \gamma_1-\imath\gamma'_1 &\text{if $\jmath=1$ and $\imath_1 = 1$,}\\ 2\gamma_1-\gamma &\text{if $\jmath=1$ and $\imath_1 = 2$,}\\ -\gamma_1-\imath\gamma'_1 &\text{if $\jmath=1$ and $\imath_1 = 3$,}\end{cases}
$$
where $\imath,\jmath\in\{0,1\}$.

\begin{proposition}\label{complemento 2 p=2} The map $(\xi_{\gamma}^1,\xi_{\gamma_1,\gamma'_1,\gamma}^2) \colon \ov{M}(2)\oplus (\ov{D}\ot \ov{M}(1))\to \Gamma$ is a $2$-cocycle of $\wh{C}_N^*(\mathcal{A},\Gamma)$. Moreover each $2$-cocycle of $\wh{C}_N^*(\mathcal{A},\Gamma)$ is cohomologous to a $(\xi_{\gamma}^1,\xi_{\gamma_1,\gamma'_1,\gamma}^2)$ and two $2$-cocycles $(\xi_{\gamma}^1,\xi_{\gamma_1,\gamma'_1,\gamma}^2)$ and $(\xi_{\ov{\gamma}}^1,\xi_{\ov{\gamma}_1,\ov{\gamma}'_1,\ov{\gamma}}^2)$ are cohomologous if and only if $\gamma_1-\ov{\gamma}_1\in 2 \Gamma$, $\gamma-\ov{\gamma}=2(\gamma_1-\ov{\gamma}_1)$ and $\ov{\gamma}'_1=\gamma'_1$.
\end{proposition}

\begin{proof} Mimic the proof of Proposition~\ref{complemento 2}.
\end{proof}

\begin{proof}[Proof of Theorem~C] This follows from Remark~\ref{son iguales2}, Proposition~\ref{complemento 2 p=2} and \cite{LV}*{Theorem~5.8}.
\end{proof}

\end{document}